\documentclass[reqno, 12pt]{amsart}

\pdfoutput=1

\usepackage{latexsym}
\usepackage[centertags]{amsmath}
\usepackage{amsfonts}
\usepackage{amssymb}
\usepackage{amsthm}
\usepackage{newlfont}
\usepackage{graphics}
\usepackage{color}
\usepackage{float}
\usepackage{diagbox}
\usepackage{longtable}
\usepackage{rotating}
\usepackage{multirow}

\usepackage{extarrows}

\usepackage[sort,compress,numbers]{natbib}

\newtheorem{thm}{Theorem}

\newtheorem{lem}[thm]{Lemma}
\newtheorem{prop}[thm]{Proposition}

\newtheorem{conj}[thm]{Conjecture}

\theoremstyle{mydefinition}

\newtheorem{rema}[thm]{Remark}

\theoremstyle{myremark}

\allowdisplaybreaks[1]

\def\pa[1]{\frac{\partial}{\partial x}}

\newcommand{\qfac}[1]{(q)_n}

\title{Hankel determinants for convolution powers of Catalan numbers}

\author{Ying Wang$^1$ and Guoce Xin$^{2,*}$}

\address{ $^{1,2}$School of Mathematical Sciences, Capital Normal University,
 Beijing 100048, PR China}
\email{$^1$\texttt{wangying.cnu@gmail.com}\ \  \& $^2$\texttt{guoce.xin@gmail.com}}
\date{November 1st, 2018}

\begin{document}

\maketitle

\begin{abstract}
The Hankel determinants $\left(\frac{r}{2(i+j)+r}\binom{2(i+j)+r}{i+j}\right)_{0\leq i,j \leq n-1}$ of the convolution powers of Catalan numbers were considered by Cigler and by Cigler and Krattenthaler. We
evaluate these determinants for $r\le 31$ by finding shifted periodic continued fractions, which
arose in application of Sulanke and Xin's continued fraction method. These include some of the conjectures of Cigler
 as  special cases. We also conjectured a polynomial characterization of these determinants. The same technique is used to
evaluate the Hankel determinants $\left(\binom{2(i+j)+r}{i+j}\right)_{0\leq i,j \leq n-1} $. Similar results are obtained.
\end{abstract}

\noindent
\begin{small}
 \emph{Mathematic subject classification}: Primary 15A15; Secondary 05A15, 11B83.
\end{small}

\noindent
\begin{small}
\end{small}

\newcommand\con[1]{\equiv_{#1}}
\section{Introduction}
Hankel determinants evaluation has a long history.
This paper is along the line of using generating functions to deal with Hankel determinants.
The Hankel determinants of a generating function $A(x)=\sum_{n\geq 0}a_nx^n$ is defined by
 $$ H_n(A(x)) = \det ( a_{i+j} )_{0\le i,j\le n-1}, \qquad H_0(A(x))=1.$$

In recent years, a considerable amount of work has been devoted to Hankel determinant evaluations of various counting coefficients.
Many of such Hankel determinants have attractive compact closed formulas,  such as
binomial coefficients, Catalan numbers \cite{J.M.E.Mays J.Wojciechowski}, Motzkin numbers \cite{M. Aigner,J.Cigler}, large and little Shr\"oder numbers \cite{R. A Brualdi and S. Kirkland.}.
See
 \cite{M. Aigner Catalan-like, M. Aigner Catalan, P. Barry,P. Barry Catalan, D.M. Bressoud, R. A Brualdi and S. Kirkland., M. Elouafi, I. Gessel and G. Viennot., Q.-H. Hou-A. Lascoux-Y.-P. Mu,C. Krattenthaler.1999, C. Krattenthaler.,C. Krattenthaler.2010,Mu.Lili-Wang.Yi-Yeh.Yeong-Nan,Mu.Lili-Wang.Yi, R.Sulanke and Xin, U. Tamm} for further references.

Classical method of continued fractions, either by $J$-fractions (Krattenthaler \cite{C. Krattenthaler.} or Wall \cite{H. S. Wall}),
or by $S$-fractions (Jones and Thron \cite[Theorem 7.2]{W. B. Jones and W. J. Thron}), requires $H_n(A(x))\neq 0$ for all $n$.
Gessel-Xin's \cite{Gessel and Xin} continued fraction method allows $H_n(A(x))=0$ for some values of $n$. Their method is based on three rules
about two variable generating functions that can transform one set of determinants to another set of determinants of the same values.
These rules corresponding to a sequence of elementary row or column operations.
This method was systematically used by Sulanke-Xin \cite{R.Sulanke and Xin} for evaluating Hankel determinants of quadratic generating functions, such as
known results for
Catalan numbers, Motzkin numbers, Shr\"oder numbers, etc. %
They defined a quadratic transformation $\tau$ such that $H(F(x))$ and $H(\tau(F(x)))$ have simple connections. This method has many applications, including
the Hankel determinants representation of Somos-4 sequence \cite{X. K. Chang-X. B. Hu and G. Xin,Xin Somos4}, and is now called Sulanke-Xin's continued fraction method.

Recently, shifted periodic continued fractions (of order $q$) of the form
 $$F_0^{(p)}(x) \mathop{\longrightarrow}\limits^\tau F_1^{(p)}(x)\mathop{\longrightarrow}\limits^\tau \cdots \mathop{\longrightarrow}\limits^\tau F_q^{(p)}(x)=F_0^{(p+1)}(x) $$
was found in \cite{Y. Wang and G. Xin} to appear in Hankel determinants of many path counting numbers. Here $p$ is an additional parameter. If one can guess an explicit formula of $F_0^{(p)}(x)$,
then their Hankel determinants can be easily computed.

One of our main objectives in this paper is to evaluate the Hankel determinants
$$H_n(C(x)^r)= \det(C_n^{(r)})_{0\le i,j\le n-1}= \det\left(\frac{r}{2n+r}\binom{2n+r}{n}\right)_{0\le i,j\le n-1}$$
of the $r$-th convolution power of the well-known Catalan numbers $C_n=\displaystyle \frac1{n+1}\binom{2n}{n}$.
The name comes after the
generating function identity
\begin{align}
  \sum_{n\ge 0} C_n^{(r)} x^n = C(x)^r = \Big(\sum_{n\ge 0} C_n x^n \Big)^r.
\end{align}

These determinants were considered by Cigler in \cite{J.Cigler-Catalan}, where the cases $r\le 4$ were evaluated by using the method of orthogonal polynomials.
In what follows, we will denote by $F(x,r)=C(x)^r$.
The cases $r=1,2$ are well-known.
\begin{align}
 H_n(F(x,1))&= H_n(C(x))=\det(C_{i+j})_{0\le i,j\le n-1}=   1, \\
H_n(F(x,2))&= H_n((C(x)-1)/x)= \det(C_{i+j+1})_{0\le i,j\le n-1}=   1.
\end{align}

The cases $r=3,4$ are nice.
\begin{thm}\label{example-F3}\cite{J.Cigler-Catalan,J.Cigler-C.Krattenthaler}
\begin{align*}
 H_{3n}(F(x,3))=H_{3n+1}(F(x,3))=(-1)^n,\quad H_{3n+2}(F(x,3))=0.
\end{align*}
\end{thm}

\begin{thm}\cite{J.Cigler-Catalan}\label{example-F4}
\begin{align*}
 H_{2n}(F(x,4))=H_{2n+1}(F(x,4))=(-1)^n(n+1).
\end{align*}
\end{thm}

Besides some byproducts, Cigler observed that such determinants follow a modular pattern and are in some cases easy to guess.
He made conjectures for $5 \leq r\leq 8$ and find that the pattern for odd $r$ is quite different from that of even $r$.
He thought $D(n,8)$ (equal to our $H_n(F(x,8))$) was already complicated.
\begin{thm}\label{example-F5}
\begin{align*}
 H_{5n}(F(x,5))=&H_{5n+1}(F(x,5))=1, \\
 H_{5n+2}(F(x,5))=&- H_{5n+4}(F(x,5))=-5(n+1),\\
 H_{5n+3}(F(x,5))=&0.
\end{align*}
\end{thm}
\begin{thm}\label{example-F7}
\begin{align*}
 H_{7n}(F(x,7))=&H_{7n+1}(F(x,7))=(-1)^n, \\
 H_{7n+2}(F(x,7))=&(-1)^n\left(\frac76\, \left( n+1 \right)  \left( 98\,{n}^{2}+49\,n-12 \right)\right),\\
 H_{7n+3}(F(x,7))=&-H_{7n+5}(F(x,7))=-(-1)^n\left(7(n+1)\right)^2,\\
 H_{7n+4}(F(x,7))=&0,\\
 H_{7n+6}(F(x,7))=&(-1)^n\left(\frac76\, \left( n+1 \right)  \left( 98\,{n}^{2}+343\,n+282 \right) \right).
\end{align*}
\end{thm}

\begin{thm}\label{example-F6}
\begin{align*}
  H_{3n}\left(F(x,6)\right)= &  H_{3n+1}\left(F(x,6)\right)=(-1)^n(n+1)^2, \\
   H_{3n+2}\left(F(x,6)\right)=&(-1)^{n+1}\left(\frac32\, \left( n+2 \right)  \left( n+1 \right)  \left( 2\,n+3 \right) \right).
\end{align*}
\end{thm}

\begin{thm}\label{example-F8}
\begin{align*}
  H_{4n}\left(F(x,8)\right)= &  H_{4n+1}\left(F(x,8)\right)=(n+1)^3, \\
   H_{4n+2}\left(F(x,8)\right)=&\frac2{45}\cdot (n+1)^2(n+2)(2n+3)(64n^2+32n-75),\\
   H_{4n+3}\left(F(x,8)\right)=&-\frac2{45}\cdot (n+1)(n+2)^2(2n+3)(64n^2+352n+405) .
\end{align*}
\end{thm}

For general $r$, he only conjectured nice formulas for
special values of $n$.
\begin{conj}\cite{J.Cigler-Catalan}\label{Cigler-conj-odd}
For odd positive integer $r=2t+1$, we have
\begin{align*}
  &  H_{(2t+1)n}(F(x,2t+1))=  H_{(2t+1)n+1}(F(x,2t+1))=(-1)^{tn}, \\
  &  H_{(2t+1)n+t+1}(F(x,2t+1))=0,\\
  &  H_{(2t+1)n+t}(F(x,2t+1))=- H_{(2t+1)n+t+2}(F(x,2t+1))=(-1)^{tn+\binom{t}2}((2t+1)(n+1))^{t-1},\\
  &  H_{(2t+1)n-1}(F(x,2t+1))+ H_{(2t+1)n+2}(F(x,2t))=(-1)^{tn+1}(t-1)(2t+1).
\end{align*}
\end{conj}
\begin{conj}\cite{J.Cigler-Catalan}\label{Cigler-conj-even}
  For even positive integer $r=2t$, we have
  \begin{align*}
   &  H_{tn}(F(x,2t))= H_{tn+1}(F(x,2t))=(-1)^{n\binom{t}{2}}(n+1)^{t-1}, \\
     &  H_{2tn-1}(F(x,2t))+ H_{2tn+2}(F(x,2t))=-t(2t-3)(2n+1)^{t-1}.
  \end{align*}
\end{conj}

We find shifted periodic continued fractions is perfect for evaluating $H_n(C(x)^r)$ for $r$ up to 31, thereby confirming
some conjectured formulas of Cigler. Though we are not able to prove Conjectures  \ref{Cigler-conj-odd} and \ref{Cigler-conj-even} at this moment, we come up with the following
conjectures.

%

\begin{conj}\label{conj-odd}
  For odd positive integer $r=2t +1$, we have,
for $1 \leq j\leq t $, $(-1)^{tn}H_{rn+j}(F(x,r))$ and $(-1)^{tn} H_{rn+ r+1-j}(F(x,r))$ are both polynomials in $n$ of degree $(j-1)(r -2j)$.
\end{conj}

\begin{conj}\label{conj-even}
  For even positive integer $r=2t$, we have,
  for $1 \leq j\leq t $, $(-1)^{n\binom{t}{2}}H_{tn+j}(F(x,r))$ and $(-1)^{n\binom{t}{2}}H_{tn+t-j+1}(F(x,r))$ are both polynomials in $n$ of degree $(2j-1)(t-j)$.
\end{conj}

Our main result is the following.
\begin{thm}
  Conjectures \ref{Cigler-conj-odd}, \ref{Cigler-conj-even}, \ref{conj-odd} and \ref{conj-even} hold true for $r\le 31$.
\end{thm}

The paper is organized as follows.
Section \ref{sec:sulanke-xin-trans} introduces Sulanke-Xin's continued fraction method, especially the quadratic transformation $\tau$.
In Section \ref{subsec-func-F} we explain how to find the functional equation of $F(x,r)=C(x)^r$ for specific $r$. We can then guess explicit functional equations of $F(x,r)$ for
general $r$. Details are put in Appendix \ref{f-e-proof}.
Section \ref{sec-HF-odd} devotes to evaluating the Hankel determinants of $F(x,r)$ for odd $r$ up to $9$. The mail tool is Sulanke-Xin's continued fraction method.
We give detailed steps on the transformation $\tau$ for small $r$.
Section \ref{sec-HF-even} deals with the even $r$ case for $r\le 8$ in a similar way. For both odd $r$ and even $r$, the formulas suggests a polynomial characterization
of the Hankel determinants, as stated in Conjectures \ref{conj-odd} and \ref{conj-even}.
In Section \ref{sec-HG}, we consider the Hankel determinants $\left(\binom{2(i+j)+r}{i+j} \right)_{0\le i,j\le n-1}$, which have been studied in \cite{ J.Cigler-Mike,Omer Egecioglu, Omer Egecioglu 0}.
The corresponding generating function is $G(x,r)=\frac{C(x)^r}{1-2xC(x)}$.
We obtain similar results as for $F(x,r)$. Finally, Appendix \ref{s-func-eq} includes the explicit functional equations of $F(x,r)$ and $G(x,r)$, as well as their proofs.

\section{Preliminary}

We have two goals in this section. Firstly we will introduce the continued fraction method of Sulanke and Xin, especially their quadratic transformation $\tau$ in \cite{R.Sulanke and Xin}.
This is the main tool of this paper. Secondly we explain how to find the functional equation of $F(x,r)=C(x)^r$, which is the starting point
for applying the transformation $\tau$.

\subsection{Sulanke-Xin's quadratic transformation $\tau$\label{sec:sulanke-xin-trans}}
This subsection is copied from \cite{Y. Wang and G. Xin}. We include it here for reader's convenience.

Suppose the generating function $F(x)$ is the unique solution of a quadratic functional equation which can be written as
\begin{gather}
  F(x)=\frac{x^d}{u(x)+x^kv(x)F(x)},\label{xinF(x)}
\end{gather}
where $u(x)$ and $v(x)$ are rational power series with nonzero constants, $d$ is a nonnegative integer, and $k$ is a positive integer.
We need the unique decomposition of $u(x)$ with respect to $d$: $u(x)=u_L(x)+x^{d+2}u_H(x)$ where $u_L(x)$ is a polynomial of degree at most $d+1$ and $u_H(x)$ is a power series.
Then Propositions 4.1 and 4.2 of \cite{R.Sulanke and Xin} can be summarized as follows.
\begin{prop}\label{xinu(0,i,ii)}
Let $F(x)$ be determined by  \eqref{xinF(x)}. Then the quadratic transformation $\tau(F)$ of $F$ defined as follows gives close connections
between   $H(F)$ and $H(\tau(F))$.
\begin{enumerate}
\item[i)] If $u(0)\neq1$, then $\tau(F)=G=u(0)F$ is determined by $G(x)=\frac{x^d}{u(0)^{-1}u(x)+x^ku(0)^{-2}v(x)G(x)}$, and $H_n(\tau(F))=u(0)^{n}H_n(F(x))$;

\item[ii)] If $u(0)=1$ and $k=1$, then $\tau(F)=x^{-1}(G(x)-G(0))$, where $G(x)$ is determined by
$$G(x)=\frac{-v(x)-xu_L(x)u_H(x)}{u_L(x)-x^{d+2}u_H(x)-x^{d+1}G(x)},$$
and we have
$$H_{n-d-1}(\tau(F))=(-1)^{\binom{d+1}{2}}H_n(F(x));$$

\item[iii)] If $u(0)=1$ and $k\geq2$, then $\tau(F)=G$, where $G(x)$ is determined by
$$G(x)=\frac{-x^{k-2}v(x)-u_L(x)u_H(x)}{u_L(x)-x^{d+2}u_H(x)-x^{d+2}G(x)},$$
and we have
$$H_{n-d-1}(\tau(F))=(-1)^{\binom{d+1}{2}}H_n(F(x)).$$\label{xinu(0)(ii)}

\end{enumerate}

\end{prop}

\subsection{The functional equation of $F(x,r)$}\label{subsec-func-F}
Our starting point is the well-known functional equation for Catalan generating function.
\begin{align}
  \label{e-func-C(x)}
  C(x)=1+x C(x)^2 \quad\Leftrightarrow \quad xC(x)^2-C(x)+1=0\quad\Leftrightarrow \quad C(x)=\frac{1}{1-xC(x)}.
\end{align}

A classical theory in Algebra asserts that for any rational function $R(x)$,
$R(C(x))$ can be uniquely written as $\alpha(x) + \beta(x) C(x)$ for some rational functions $\alpha(x)$ and
$\beta(x)$. Hence $R(C(x))$ satisfies a quadratic functional equation, so does $F(x,r)=C(x)^r$.

Using the method of undetermined coefficients, one can easily obtain the functional equation of $F(x,r)$ with the help of Maple.
For odd integer $r$, the first several functional equations are as follows.
\begin{gather*}
 F(x,3) =- \frac1{{x}^{3}\,F(x,3)+3\,x-1 }, \\
  F(x,5)=- \frac1{{x}^{5}\,F(x,5)-5\,{x}^{2}+5\,x-1 },\\
  F(x,7)=- \frac1{{x}^{7}\,F(x,7)+7\,{x}^{3}-14\,{x}^{2}+7\,x-1 },\\
  F(x,9)=- \frac1{{x}^{9}\,F(x,9)-9\,{x}^{4}+30\,{x}^{3}-27\,{x}^{2}+9\,x-1 },\\
 F(x,11)=- \frac1{{x}^{11}\,F(x,11)+11\,{x}^{5}-55\,{x}^{4}+77\,{x}^{3}-44\,{x}^{2}+11
\,x-1 }.
\end{gather*}
For even integer $r$, the first several functional equations are as follows.
\begin{gather*}
 F(x,2) =- \frac1{{x}^{2}\,F(x,2)+2\,x-1 }, \\
  F(x,4)=- \frac1{{x}^{4}\,F(x,4)-2\,{x}^{2}+4\,x-1 },\\
  F(x,6)=- \frac1{{x}^{6}\,F(x,6)+2\,{x}^{3}-9\,{x}^{2}+6\,x-1 },\\
  F(x,8)=- \frac1{{x}^{8}\,F(x,8)-2\,{x}^{4}+16\,{x}^{3}-20\,{x}^{2}+8\,x-1 },\\
 F(x,10)=- \frac1{{x}^{10}\,F(x,10)+2\,{x}^{5}-25\,{x}^{4}+50\,{x}^{3}-35\,{x}^{2}+10\,
x-1 }.
\end{gather*}
Though these formula are sufficient for our purpose, the above formulas are so nice that we guess and prove
explicit functional equation of $F(x,r)$ for general $r$. See Appendix \ref{f-e-proof} for details.

\section{The results of $H_{n}(F(x,r))$ for odd $r=2t +1$}\label{sec-HF-odd}
In this section, we evaluate the Hankel determinants of $F(x,r)=C(x)^r$ for odd $r=2t+1$.

\subsection{The results of $F(x,3)$}
For $r=3$, we have the functional equation
\begin{align*}
   F(x,3)=\frac{1}{1-3x -{x}^{3}F(x,3) }.
\end{align*}
\begin{proof}[Proof of Theorem \ref{example-F3}]
Let $F_0(x)=F(x,3)$. We apply Proposition \ref{xinu(0,i,ii)} to obtain $F_1=\tau (F_0)$.
Firstly, $d=0$. $u(x)=1-3x$, thus $u_L(x)=1-3x$ and $u_H(x)=0$. Then by $u(0)^{-1}=1$, we obtain
\begin{align*}
   H_k(F_0)&=H_{k-1}(F_1) \ \ &F_1(x)=\frac{x}{1-3x -{x}^{2} F_1(x)}.
\end{align*}
Apply Proposition \ref{xinu(0,i,ii)} to obtain $F_2=\tau (F_1)$.
This time $d=1$. $u(x)=1-3x$, thus $u_L(x)=1-3x$ and $u_H(x)=0$. Then by $u(0)^{-1}=1$, we obtain
\begin{align*}
   H_{k-1}(F_1)&=(-1)^{\binom{2}{2}}H_{k-3}(F_2)=-H_{k-3}(F_2) \ \ &F_2(x)=\frac{1}{1-3x -{x}^{3} F_2(x)}.
\end{align*}
Now we see that $F_2=F_0$. So the continued fractions is periodic of order $2$: $F_0 \to F_1 \to F_2=F_0.$
By summarizing the above results, we obtain  the recursion $H_k(F_0)=-H_{k-3}(F_0)$, which implies  that
$$H_{3n+i}(F_0)=(-1)^{n}H_i(F_0),\quad \text{ for } i=0,1,2 .$$

The initial values are $H_0(F_0)=1,\ H_1(F_0)=1,\ H_2(F_0)=0.$ Theorem \ref{example-F3} then follows.
\end{proof}
\subsection{The results of $F(x,5)$}For $r=5$, we have the functional equation
\begin{align*}
   F(x,5)=\frac{1}{1-5x +5x^2-{x}^{5}F(x,5) }.
\end{align*}
\begin{proof}[Proof of Theorem \ref{example-F5}]
We apply Proposition \ref{xinu(0,i,ii)} to $F_0:=F(x,5)$ by repeatedly using the transformation $\tau$. This results in a shifted periodic continued fractions of order $4$:
$$F_0(x) \mathop{\longrightarrow}\limits^\tau F^{(1)}_1(x)\mathop{\longrightarrow}\limits^\tau F^{(1)}_2(x)\mathop{\longrightarrow}\limits^\tau F^{(1)}_3(x)\mathop{\longrightarrow}\limits^\tau F^{(1)}_4(x)\mathop{\longrightarrow}\limits^\tau F^{(1)}_5(x)= F^{(2)}_{1}(x) \cdots.$$
We will carry out the details in this computation.

Apply Proposition \ref{xinu(0)(ii)} to obtain $F_1^{(1)}=\tau (F_0)$. We obtain
\begin{align}
  H_k(F_0)=H_{k-1}(F_1^{(1)}).\label{e-5F0}
\end{align}
Computer experiment suggests us to define, for $p\geq1$,
$$F^{(p)}_{1}=-{\frac {{x}^{3}+5^2\,p \left( p-1 \right) {x}^{2}+ 5^2\,p\,x-5\,p}{x^2\, F^{(p)}_{1}+ 5\left(2\,p-1 \right) {x}^{2}+5\,x-1}}.$$
Apply Proposition \ref{xinu(0,i,ii)} to obtain $F^{(p)}_{2}=\tau (F^{(p)}_{1})$.
Firstly, $d=0$. Next we need to decompose $u(x)$ with respect to $d$.
We expand $u(x)$ as a power series and focus on (by displaying) those terms  with small exponents ($\le d+1=1$):
$$u(x)=-\frac{5\left(2\,p-1 \right) {x}^{2}+5\,x-1}{{x}^{3}+5^2\,p \left( p-1 \right) {x}^{2}+ 5^2\,p\,x-5\,p}
=-\frac1{5p}+{x}^{2}+\frac1{25}\,{\frac {125\,{p}^{2}-1}{{p}^{2}}}{x}^{3}
+\cdots
,$$
Thus
$u_L(x)=\displaystyle-\frac1{5p}$ is simple, and
$ u_H(x)=\displaystyle-\frac15\,{\frac {25\,{p}^{2}-x}{ \left( {x}^{3}+25\,p \left( p-1 \right)
{x}^{2}+25\,px-5\,p \right) p}}.$
Then by $u(0)^{-1}=-5p$, we obtain
\begin{align*}
 & H_{k-1}(F^{(p)}_{1})=\left(-5\,p\right)^{k-1}H_{k-2}(F^{(p)}_{2}),\\
  &F^{(p)}_{2}=-{\frac {x}{{x}^{2} \left( {x}^{3}+5^2\,p \left( p-1 \right) {x}^{2}+5^2\,p x-5\,p\right) F^{(p)}_{2}-2\,{x}^{3}+5^2\,p{x}^{2}-5^2\,xp+5\,p}}.
\end{align*}
Apply Proposition  \ref{xinu(0,i,ii)} to obtain $F^{(p)}_{3}=\tau (F^{(p)}_{2})$,
This time $d=1$ and $u(x)$ is indeed a polynomial: $$u(x)=2\,{x}^{3}-5^2\,p{x}^{2}+5^2\,xp-5\,p.$$
It then follows that $u_L(x)=-5^2\,p{x}^{2}+5^2\,xp-5\,p,$
$u_H(x)=2\,{x}.$
Then by $u(0)^{-1}=-\dfrac{1}{5\,p}$, we obtain
\begin{align*}
&H_{k-2}(F^{(p)}_{2})=(-1)^{\binom{2}{2}}\left(-\frac{1}{5\,p}\right)^{k-2}H_{k-4}(F^{(p)}_{3}),\\
&F^{(p)}_{3}=-\frac15\cdot\,{\frac {{x}^{3}+5^2\,p \left( p+1 \right) {x}^{2}-5^2\,xp+5\,p}{
 \left( 5\,p{x}^{3}F^{(p)}_{3}-2\,{x}^{3}-5^2\,p{x}^{2}+5^2\,xp-5\,p \right) p}}.
\end{align*}
Apply Proposition  \ref{xinu(0,i,ii)} to obtain $F^{(p)}_{4}=\tau (F^{(p)}_{3})$,
This time $d=0$ and $u(x)$ is expanded as:
$$u(x)=-\frac{5(-2\,{x}^{3}-5^2\,p{x}^{2}+5^2\,xp-5\,p)p}{{x}^{3}+5^2\,p \left( p+1 \right) {x}^{2}-5^2\,xp+5\,p}=5\,p-25\,{p}^{2}{x}^{2}+\cdots.$$
It then follows that $u_L(x)=5\,p,$
$u_H(x)=-5\,{\dfrac {p \left( 25\,{p}^{2}-x \right) }{{x}^{3}+25\,p \left( p+1
 \right) {x}^{2}-25\,px+5\,p}}
.$
Thus by $u(0)^{-1}=\dfrac{1}{5\,p}$, we obtain
\begin{align*}
   &H_{k-4}(F^{(p)}_{3})=\left(\frac{1}{5\,p}\right)^{k-4}H_{k-5}(F^{(p)}_{4}),\\
 & F^{(p)}_{4}=-{\frac {{5^2p^2}}{{x}^{2} \left({x}^{3}+5^2\,p \left( p+1 \right) {x}^{2}-5^2\,p x+5\,p\right) F^{(p)}_{4}-5\,p \left( 5(2\,p+1){x}^{2}-5\,x+
1 \right) }}
.
\end{align*}
Apply Proposition  \ref{xinu(0,i,ii)} to obtain $F^{(p)}_{5}=F^{(p+1)}_{1}=\tau (F^{(p)}_{4})$,
This time $d=0$ and $u(x)$ is indeed a polynomial:
$$u(x)=\frac{5(2\,p+1){x}^{2}-5\,x+
1 }{5\,p}.$$
It then follows that $u_L(x)=\dfrac{-5\,x+
1}{5\,p},$
$u_H(x)=\dfrac {2\,p+1 }{p}
.$
Then by $u(0)^{-1}=u(0)^{-1}=5p$, we obtain
\begin{align*}
  H_{k-5}(F^{(p)}_{4})&=(5p)^{k-5}H_{k-6}(F^{(p+1)}_{1}).
\end{align*}
Combining the above formulas gives the recursion
\[H_{k-1}(F^{(p)}_{1})=H_{k-6}(F^{(p+1)}_{1}).\]
Let $k-1=5n+j$, where $0\leq j< 5$. We then deduce that
\begin{gather}
  H_{5n+j}(F_1^{(1)})=H_{j}(F^{(n+1)}_{1}).\label{e-5F5n}
\end{gather}
The initial values are
\begin{gather*}
 H_{0}(F^{(n+1)}_{1})=1,H_{1}(F^{(n+1)}_{1})=-H_{3}(F^{(n+1)}_{1})=-5(n+1),H_{2}(F^{(n+1)}_{1})=0,H_{4}(F^{(n+1)}_{1})=1.
\end{gather*}
The theorem then follows by the above initial values, \eqref{e-5F0} and \eqref{e-5F5n}.
\end{proof}

\subsection{The results of $H_n(F(x,7))$}For $r=7$, we have the functional equation
\begin{align*}
   F(x,7)=\frac1{ 1- 7\,x+14\,{x}^{2} -7\,{x}^{3}-{x}^{7}F(x,7) }.
\end{align*}
\begin{proof}[Proof of Theorem \ref{example-F7}]
We apply Proposition \ref{xinu(0,i,ii)} to $F_0:=F(x,7)$, and repeat the transformation $\tau$.
This results in a shifted periodic continued fractions of order $6$:
\begin{align*}
  F_0(x) &\mathop{\longrightarrow}\limits^\tau F^{(1)}_1(x)\mathop{\longrightarrow}\limits^\tau F^{(1)}_2(x)\mathop{\longrightarrow}\limits^\tau F^{(1)}_3(x)\mathop{\longrightarrow}\limits^\tau F^{(1)}_4(x)\mathop{\longrightarrow}\limits^\tau F^{(1)}_5(x)\mathop{\longrightarrow}\limits^\tau F^{(1)}_6(x)\\& \mathop{\longrightarrow}\limits^\tau F^{(1)}_7(x)= F^{(2)}_{1}(x)\cdots.
\end{align*}
If we define
\begin{align*}
 F_1^{(p)}&=-[{x}^{5}+49\, \left( p-1 \right) p{x}^{4}+{ {49}/{6}}\,
 \left( p-1 \right) p \left( 196\,p ^{2}-196\,p+25
 \right) {x}^{3}+{ {49}/{36}}\,p (9604\,{p}^{5}-28812\,{p}^{4}\\
 &+28861\,{p}^{3}-9702\,{p}^{2}-395\,p+408) {x}^{2}-7/6\,p \left(
686\,{p}^{2}-1029\,p+253 \right) x+7/6\,p ( 98\,{p}^{2}\\
&-147\,p+37 )]/[F_1^{(p)}{x}^{2}-7
\, \left( 2\,p-1 \right) {x}^{3}-7/3\, \left( 2\,p-1 \right)(
49\,{p}^{2}-49\,p-6) {x}^{2}+7\,x-1 ],
\end{align*}
Then the results can be summarized as follows:
\begin{align}
  H_k(F_0)=H_{k-1}(F_1^{(1)}),\label{e-7F0}
\end{align}
and
\begin{align*}
  H_{k-1}(F^{(p)}_{1})=&\left(\frac76\cdot \,p \left( 98\,{p}^{2}-147\,p+37 \right)\right)^{k-1} H_{k-2}(F^{(p)}_{2}), \\
H_{k-2}(F^{(p)}_{2})=&\left(-\frac{36}{\left( 98\,{p}^{2}-147\,p+37 \right)^{2}}\right)^{k-2}H_{k-3}(F^{(p)}_{3}),\\
 H_{k-3}(F^{(p)}_{3})=&(-1)^{\binom{2}{2}}\left(-\frac1{42}\cdot\,{\frac {98\,{p}^{2}-147\,p+37}{p}}\right)^{k-3}H_{k-5}(F^{(p)}_{4}),\\
 H_{k-5}(F^{(p)}_{4})=&\left(\frac 1{42}\cdot\,{\frac {98\,{p}^{2}+147\,p+37}{p}}\right)^{k-5} H_{k-6}(F^{(p)}_{5}), \\
H_{k-6}(F^{(p)}_{5})=&\left(-\frac{36}{\left( 98\,{p}^{2}+147\,p+37 \right)^{2}}\right)^{k-6}H_{k-7}(F^{(p)}_{6}),\\
 H_{k-7}(F^{(p)}_{6})=&\left(-\frac76\cdot\,p \left( 98\,{p}^{2}+147\,p+37 \right) \right)^{k-7}H_{k-8}(F^{(p+1)}_{1}).
\end{align*}
Combining the above formulas gives the recursion
\[H_{k-1}(F_1^{(p)})=-H_{k-8}(F^{(p+1)}_{1}).\]
Let $k-1=7n+j$, where $0\leq j< 7$. We then deduce that
\begin{gather}
  H_{7n+j}(F_1^{(1)})=(-1)^nH_{j}(F^{(n+1)}_{1}).\label{e-7F7n}
\end{gather}
The initial values are
\begin{align*}
& H_{0}(F^{(n+1)}_{1})=1,\ H_{1}(F^{(n+1)}_{1})=\frac76\, \left( n+1 \right)  \left( 98\,{n}^{2}+49\,n-12 \right) ,\ H_{2}(F^{(n+1)}_{1})=-49\, \left( n+1 \right) ^{2},\\
 &H_{3}(F^{(n+1)}_{1})=0,\ H_{4}(F^{(n+1)}_{1})=49\, \left( n+1 \right) ^{2},\ H_{5}(F^{(n+1)}_{1})=\frac76\, \left( n+1 \right)  \left( 98\,{n}^{2}+343\,n+282 \right),\\
 & H_{6}(F^{(n+1)}_{1})=1.
\end{align*}
The theorem then follows by the above initial values, \eqref{e-7F0} and \eqref{e-7F7n}.
\end{proof}
\subsection{The results of $H_n(F(x,9))$}For $r=9$, we have the functional equation
\begin{align*}
   F(x,9)=\frac1{1-9\,x+27\,{x}^{2}-30\,{x}^{3}+9\,{x}^{4}-{x}^{9}F(x,9)}.
\end{align*}

\begin{thm}
\begin{align*}
 H_{9n}(F(x,9))=&H_{9n+1}(F(x,9))=1, \\
 H_{9n+2}(F(x,9))=&-{\frac {27}{10}}\, \left( 3\,n+2 \right)  \left( 18\,n+1 \right)
 \left( n+1 \right)  \left( 54\,{n}^{2}+42\,n+5 \right),\\
 H_{9n+3}(F(x,9))=&{\frac {9}{20}}\, \left( n+1 \right) ^{2} \left( 26244\,{n}^{4}+104976
\,{n}^{3}+108459\,{n}^{2}+31266\,n-1460 \right),\\
 H_{9n+4}(F(x,9))=&-H_{9n+6}(F(x,9))=\left(9( n+1) \right) ^{3},\  H_{9n+5}(F(x,9))=0,\\
 H_{9n+7}(F(x,9))=&{\frac {9}{20}}\, \left( n+1 \right) ^{2} \left( 26244\,{n}^{4}+104976
\,{n}^{3}+108459\,{n}^{2}-17334\,n-50060 \right),\\
 H_{9n+8}(F(x,9))=&{\frac {27}{10}}\, \left( 18\,n+35 \right)  \left( 3\,n+4 \right)
 \left( n+1 \right)  \left( 54\,{n}^{2}+174\,n+137 \right) .
\end{align*}
\end{thm}
\begin{proof}
We apply Proposition \ref{xinu(0,i,ii)} to $F_0:=F(x,9)$, and repeat the transformation $\tau$.
This results in a shifted periodic continued fractions of order $8$:
\begin{align*}
  F_0(x) &\mathop{\longrightarrow}\limits^\tau F_1^{(1)}(x)\mathop{\longrightarrow}\limits^\tau F^{(1)}_{2}(x)\mathop{\longrightarrow}\limits^\tau F^{(1)}_{3}(x)\mathop{\longrightarrow}\limits^\tau F^{(1)}_{4}(x)\mathop{\longrightarrow}\limits^\tau F^{(1)}_{5}(x)\mathop{\longrightarrow}\limits^\tau F^{(1)}_{6}(x)\\& \mathop{\longrightarrow}\limits^\tau F^{(1)}_{7}(x)
  \mathop{\longrightarrow}\limits^\tau F^{(1)}_{8}(x)\mathop{\longrightarrow}\limits^\tau F_9^{(1)}(x)= F_2^{(1)}(x) \cdots.
\end{align*}
The formula of $F^{(p)}_1(x)$ is too complicated to print.
Then the results can be summarized as follows:
\begin{align}
  H_k(F_0)=H_{k-1}(F_1^{(1)}),\label{e-9F0}
\end{align}
and
\begin{align*}
  H_{k-1}(F^{(p)}_{1})=&\left(-{\frac {27}{10}}\, \left( 3\,p-1 \right)  \left( 18\,p-17 \right) p
 \left( 54\,{p}^{2}-66\,p+17 \right)\right)^{k-1} H_{k-2}(F^{(p)}_{2}), \\
H_{k-2}(F^{(p)}_{2})=&\left({\frac {5}{81}}\,{\frac {26244\,{p}^{4}-49005\,{p}^{2}+24300\,p-2999}{
 \left( 3\,p-1 \right) ^{2} \left( 18\,p-17 \right) ^{2} \left( 54\,{p
}^{2}-66\,p+17 \right) ^{2}}}\right)^{k-2}H_{k-3}(F^{(p)}_{3}),\\
 H_{k-3}(F^{(p)}_{3})=&\left(-9720\,{\frac { \left( 3\,p-1 \right)  \left( 18\,p-17 \right)
 \left( 54\,{p}^{2}-66\,p+17 \right) }{ \left( 26244\,{p}^{4}-49005\,{
p}^{2}+24300\,p-2999 \right) ^{2}}}\right)^{k-3}H_{k-4}(F^{(p)}_{4}),\\
 H_{k-4}(F^{(p)}_{4})=&\displaystyle(-1)^{\binom{2}{2}}\left({\frac {1}{1620}}\,{\frac {26244\,{p}^{4}-49005\,{p}^{2}+24300\,p-2999
}{p}}\right)^{k-4} H_{k-6}(F^{(p)}_{5}), \\
H_{k-6}(F^{(p)}_{5})=&\left(-{\frac {1}{1620}}\,{\frac {26244\,{p}^{4}-49005\,{p}^{2}-24300\,p-
2999}{p}}
\right)^{k-6}H_{k-7}(F^{(p)}_{6}),\\
 H_{k-7}(F^{(p)}_{6})=&\left(-9720\,{\frac { \left( 18\,p+17 \right)  \left( 3\,p+1 \right)
 \left( 54\,{p}^{2}+66\,p+17 \right) }{ \left( 26244\,{p}^{4}-49005\,{
p}^{2}-24300\,p-2999 \right) ^{2}}}
 \right)^{k-7}H_{k-8}(F_{7}^{(p)}),\\
  H_{k-8}(F_{7}^{(p)})=&\left({\frac {5}{81}}\,{\frac {26244\,{p}^{4}-49005\,{p}^{2}-24300\,p-2999}{
 \left( 18\,p+17 \right) ^{2} \left( 3\,p+1 \right) ^{2} \left( 54\,{p
}^{2}+66\,p+17 \right) ^{2}}}
\right)^{k-8} H_{k-9}(F_{8}^{(p)}), \\
H_{k-9}(F_8^{(p)})=&\left({\frac {27}{10}}\, \left( 18\,p+17 \right)  \left( 3\,p+1 \right) p
 \left( 54\,{p}^{2}+66\,p+17 \right)
\right)^{k-9}H_{k-10}(F^{(p+1)}_{1}).
\end{align*}
Combining the above formulas gives the recursion
\[H_{k-1}(F_1^{(p)})=H_{k-10}(F^{(p+1)}_{1}).\]
Let $k-1=9n+j$, where $0\leq j< 9$. We then deduce that
\begin{gather}
  H_{9n+j}(F_1^{(1)})=H_{j}(F^{(n+1)}_{1}).\label{e-9F9n}
\end{gather}
The initial values are
\begin{align*}
 H_{0}(F^{(n+1)}_{1})=&1,\  H_{1}(F^{(n+1)}_{1})=-{\frac {27}{10}}\, \left( 3\,n+2 \right)  \left( 18\,n+1 \right)
 \left( n+1 \right)  \left( 54\,{n}^{2}+42\,n+5 \right),\\
 H_{2}(F^{(n+1)}_{1})=&{\frac {9}{20}}\, \left( n+1 \right) ^{2} \left( 26244\,{n}^{4}+104976
\,{n}^{3}+108459\,{n}^{2}+31266\,n-1460 \right),\\
H_{3}(F^{(n+1)}_{1})=& \left(9( n+1) \right) ^{3}, \quad H_{4}(F^{(n+1)}_{1})=0, \quad H_{5}(F^{(n+1)}_{1})=-\left(9( n+1) \right) ^{3},\\
  H_{6}(F^{(n+1)}_{1})=&{\frac {9}{20}}\, \left( n+1 \right) ^{2} \left( 26244\,{n}^{4}+104976
\,{n}^{3}+108459\,{n}^{2}-17334\,n-50060 \right),\\
 H_{7}(F^{(n+1)}_{1})=&{\frac {27}{10}}\, \left( 18\,n+35 \right)  \left( 3\,n+4 \right)
 \left( n+1 \right)  \left( 54\,{n}^{2}+174\,n+137 \right),\  H_{8}(F^{(n+1)}_{1})=1.
\end{align*}
The theorem then follows by the above initial values, \eqref{e-9F0} and \eqref{e-9F9n}.
\end{proof}
\subsection{A summary} A computer program has been developed for calculating $H_n(F(x,r))$ for $r=3,5,7,\dots$. We succeed for $r$ up to $31$
because the coefficients appearing in the functional equation for $F_1^{(p)}(x)$ are all polynomials in $p$ of a reasonable small degree. This allows us to guess an explicit
formula, and then the rest is straightforward.

From the obtained formulas, we find that $\pm H_{rn+i}(F(x,r))$ is a polynomial in $n$ for fixed $r$ and $i$. The degrees are listed in Table 1.
\begin{table}[H] 
  \centering  
    {\begin{tabular}{|c|c|c|c|c|c|c|c|c|c|c|c|c|c|c|c|c|c|c|c|}   
       \hline
   {\color{red}{\diagbox[dir=SE]{{\color{black}$r$}}{{\color{black}$i$}}}} & $0$ & $1$ & $2$ & $3$ & $4$ & $5$ & $6$ & $7$ & $8$ & $9$ &$10$& $11$&$12$&$13$&$14$&$15$&$16$\\
  \hline
  $r=3$ & 0 & 0 &\textbf{0} &  &  &  &  &  &  &   &  &  &  &  &  &  &      \\
  \hline
  $r=5$ & 0 & 0 & 1 & \textbf{0} & 1 &  &  &  & &    &  &  &  &  &  &  &     \\
  \hline
  $r=7$ & 0 & 0 & 3 & 2 & \textbf{0} & 2 & 3 &  &  &    & &  &  &  &  & &    \\
  \hline
  $r=9$ & 0 & 0 & 5 & 6 & 3 & \textbf{0} & 3 & 6 & 5 &  &  &  &  &  & & &\\
  \hline
  $r=11$ &0 & 0 & 7 & 10 &9 & 4 & \textbf{0} & 4 & 9 & 10  &7 &  &  &  & & &  \\
  \hline
  $r=13$ &0 & 0 & 9 & 14 &15&12 & 5 & \textbf{0} & 5 &12 &15&14 &9 &  & & & \\
  \hline
  $r=15$ &0 & 0 & 11& 18 &21&20 & 15& 6 & \textbf{0} & 6 &15&20 &21 &  18&11  &    &\\
  \hline
  $r=17$ &0 & 0 & 13& 22 &27&28 & 25& 18& 7 & \textbf{0}   & 7 &18 &25 &28&27&22&13    \\
  \hline
    \end{tabular}
    }
    \caption{Up to a sign, $H_{rn+i}(F(x,r))$ is a polynomial in $n$,
    with the corresponding degree given in the table. (The degrees in each row are palindromic centering at the bold faced numbers.)
    For instance,
    the entry in the row indexed by $(r=)11$ and the column indexed by $(i=)4$ is $9$. This means the degree in $n$ of $\pm H_{11n+4}(F(x,11))$ is $9$.  } 
\end{table}

From Table 1, we guess a formula for the degrees. This leads to Conjecture \ref{conj-odd}. If the conjecture holds true, then we can use Lagrange's interpolation formula to obtain the formula
of $H_{rn+j}(F(x,r))$. For instance, if $r=31$, then the max degree is $190$ attained at $j=11$. Thus to obtain a formula for $H_{31n+11}(F(x,31))$, it is sufficient to evaluate these determinants for $0\le n \le 190$.
For $n=190$, we need to evaluate a determinant of size 5901. At the beginning, we thought Maple can compute these determinants quickly, but it is not the case, even if we modulo a large prime.
The computation time is given in table \ref{time}, where we can see that the computation of $H_{800}(F(x,31)) \pmod{274177}$ already takes
about $117$ seconds.

\begin{table}[H] 
  \centering  
    {\begin{tabular}{|c|c|c|c|c|c|c|c|c|c|c|c|c|c|c|c|c|c|c|c|}   
       \hline
   {\color{red}{\diagbox[dir=SE]{{\color{black}$r$}}{{\color{black}$n$}}}} & $100$ & $200$ & $300$ & $400$ & $500$ & $600$ & $700$ & $800$ \\
  \hline
  $r=1$ & 0.093 & 0.639 & 2.652 & 7.160 & 16.458 & 32.479 & 59.639 & 103.615   \\
  \hline
  $r=6$ & 0.093 & 0.655 & 2.558 & 7.160 & 16.411 & 31.995 & 60.949 &  103.584 \\
  \hline
  $r=11$ & 0.093 & 0.655 & 2.542 & 7.082 & 16.317 & 31.793 & 57.860 & 97.485 \\
  \hline
  $r=16$ & 0.093 & 0.655 & 2.698 & 7.269 & 16.551 & 32.198 & 58.843 & 97.001\\
  \hline
  $r=21$ &0.093 & 0.639 & 2.636 & 7.098 & 16.348 & 31.715 & 57.813 & 103.225  \\
  \hline
  $r=26$ &0.093 & 0.639 & 2.605 & 7.051 & 16.333 & 31.793 & 58.188 & 99.060 \\
  \hline
  $r=31$ &0.093 & 0.655 & 2.558& 7.441 & 16.567 & 32.229 & 58.016 & 98.124 \\
  \hline
  $r=36$ &0.078 & 0.670 & 2.558 & 7.207 & 16.395 & 32.105 & 58.328 & 99.419 \\
  \hline
  $r=41$ &0.093 & 0.670 & 2.636 & 7.316 & 16.239 & 31.917 & 57.455 & 96.751    \\
   \hline
    \end{tabular}
    }
    \label{time}\caption{Time (in seconds) spent for calculating $H_{n}(F(x,r)) \pmod{274177}$.  } 
\end{table}

Indeed, fast evaluation of $H_{rn+i}(F(x,r)$ is still by using our transformation $\tau$. For instance when $r=31$, we start with
the functional equation of $F_0=F(x,31)$ and successively compute $F_i=\tau(F_{i-1})$. The computer can produce
$F_i$ for $i=1$ to $6000$ in $198.480$ seconds. Then we need to record the $(u_0,\ d)$'s for each $F_i$ and obtain $H_n(F(x,31))$ for $n\le 6200$ (some of the determinants are $0$).
Using this approach we are able to verify Conjecture \ref{conj-odd} for $r\le 41$.

\section{The case of even $r=2t $}\label{sec-HF-even}
The computation of the Hankel determinants for even $r$ is similar, but the formulas have a quite different pattern.

\subsection{The results of $H_n(F(x,4))$} For $r=4$, we have the functional equation
\begin{align*}
   F(x,4)=\frac{1}{1-4\,x+2\,{x}^{2}-{x}^{4} F(x,4)}.
\end{align*}

\begin{proof}[Proof of Theorem \ref{example-F4}]
We apply Proposition \ref{xinu(0,i,ii)} to $F_0:=F(x,4)$ by repeatedly using the transformation $\tau$.
This results in a shifted periodic continued fractions of order 2:
$$F_0(x) \mathop{\longrightarrow}\limits^\tau F_1^{(1)}(x)\mathop{\longrightarrow}\limits^\tau F^{(1)}_{2}(x)\mathop{\longrightarrow}\limits^\tau F^{(1)}_{3}(x)=F_1^{(2)}(x) \cdots.$$
If we define
\begin{align*}
 F_1^{(p)}&=\frac{{x}^{2}+4\, p \left( p+1 \right){x}-p \left( p+1 \right)}{p^2-4\,p^2x-2\,px^2-p^2 F_1^{(p)}},
\end{align*}
Then the results can be summarized as follows:
\begin{align}
  H_k(F_0)=H_{k-1}(F_1^{(1)}),\label{e-4F0}
\end{align}
and
\begin{align*}
  H_{k-1}(F^{(p)}_{1})=&\left(-\frac{p+1}{p}\right)^{k-1} H_{k-2}(F^{(p)}_{2}), \\
H_{k-2}(F^{(p)}_{2})=&\left(-\frac{p}{p+1}\right)^{k-2}H_{k-3}(F^{(p+1)}_{1}).
\end{align*}
Combining the above formulas gives the recursion
\[H_{k-1}(F_1^{(p)})=-\left(\frac{p+1}{p}\right)H_{k-3}(F^{(p+1)}_{1}).\]
Let $k-1=2n+j$, where $j=0,1$. We then deduce that
\begin{gather}
  H_{2n+j}(F_1^{(1)})=(-1)^n\left(n+1\right)H_{j}(F^{(n+1)}_{1}).\label{e-4F4n}
\end{gather}
The initial values are
\begin{align*}
 H_{0}(F^{(n+1)}_{1})=&1,\quad H_{1}(F^{(n+1)}_{1})=-\left(\frac{n+2}{n+1}\right).
\end{align*}
The theorem then follows by the above initial values, \eqref{e-4F0} and \eqref{e-4F4n}.
\end{proof}
\subsection{The results of $H_n(F(x,6))$}
For $r=6$, we have the functional equation
\begin{align*}
   F(x,6)=\frac{1}{1-6\,x+9\,{x}^{2}-2\,x^3-{x}^{4} F(x,6)}.
\end{align*}
\begin{proof}[Proof of Theorem \ref{example-F6}]
We apply Proposition \ref{xinu(0,i,ii)} to $F_0:=F(x,6)$ by repeatedly using the transformation $\tau$.
This results in a shifted periodic continued fractions of order 3:
\begin{align*}
  F_0(x)& \mathop{\longrightarrow}\limits^\tau F_1^{(1)}(x)\mathop{\longrightarrow}\limits^\tau F^{(1)}_{2}(x)\mathop{\longrightarrow}\limits^\tau F^{(1)}_{3}(x)\mathop{\longrightarrow}\limits^\tau F^{(1)}_{4}(x)= F_2^{(1)}(x)\cdots.
\end{align*}
We obtain
\begin{align}
  H_k(F_0)=H_{k-1}(F_1^{(1)}).\label{e-6F0}
\end{align}
For $p\geq1$. Computer experiment suggests us to define
\begin{align*}
  F^{(p)}_{1} &= ( -\frac23\,{\frac {{x}^{4}}{ \left( 2\,p+1 \right)  \left( p+1
 \right) p}}-2\,{\frac {{x}^{3} \left( p-1 \right) }{ \left( 2\,p+1
 \right) p}}-\frac12\,{\frac {{x}^{2} \left( 2\,p-3 \right)  \left( 6\,
 p^{2}+3\,p-1 \right) }{ \left( 2\,p+1 \right) p}}-
\frac43\,{\frac {x \left( 9\,p+5 \right) }{2\,p+1}}\\
&+1 )/  \large(-\frac43
\,{\frac {{x}^{3}}{ \left( p+1 \right)  \left( 2\,p+1 \right) }}+2\,{
\frac {{x}^{2} \left( 2\, p^{2}+1 \right) }{
 \left( p+1 \right)  \left( 2\,p+1 \right) }}+4\,{\frac {xp}{ \left( p
+1 \right)  \left( 2\,p+1 \right) }}-\frac23\,{\frac {p}{ \left( p+1
 \right)  \left( 2\,p+1 \right) }}\\
 &+\frac23\,{\frac {{x}^{2}F^{(p)}_{1}p}{ \left( p+1
 \right)  \left( 2\,p+1 \right) }}\large).
 \end{align*}
We shall mention a difficulty in guessing the above formula. At the beginning, we failed because we assumed the coefficients to be polynomials in $p$, just like the odd $r$ case, but
there is a cancellation for those coefficients. For instance, the expected $F_7$ is $-\frac1{12}\,{\frac {4\,{x}^{4}+96\,{x}^{3}+2232\,{x}^{2}+3072\,x-504}{3x^2\,F_7-2\,{
x}^{3}+ 57 {x}^{2}+18\,x-3}}$, but Maple will give the cancelled form $F_7=-\frac13\,{\frac {{x}^{4}+24\,{x}^{3}+558\,{x}^{2}+768\,x-126}{3{x}^{2}\,F_7
-2\,{x}^{3}+57\,{x}^{2}+18\,x-3}} $.
After a little thought, we tried rational functions instead and succeed.

Then the results can be summarized as follows:
\begin{align*}
  H_{k-1}(F^{(p)}_{1})=&\left(-\frac32\,\cdot{\frac { \left( p+1 \right)  \left( 2\,p+1 \right) }{p}}\right)^{k-1} H_{k-2}(F^{(p)}_{2}), \\
H_{k-2}(F^{(p)}_{2})=&\left(-\frac49\, \cdot\frac1{\left( 2\,p+1 \right)^{2}}\right)^{k-2}H_{k-3}(F^{(p)}_{3}),\\
 H_{k-3}(F^{(p)}_{3})=&\left(\frac32\,\cdot{\frac { \left( 2\,p+1 \right) p}{p+1}}\right)^{k-3}H_{k-4}(F^{(p+1)}_{1}).
\end{align*}
Combining the above formulas gives the recursion
\[H_{k-1}(F_1^{(p)})=-\left(\frac{p+1}{p}\right)^2H_{k-4}(F^{(p+1)}_{1}).\]
Let $k-1=3n+j$, where $0\leq j< 3$. We then deduce that
\begin{gather}
  H_{3n+j}(F_1^{(1)})=(-1)^n\left(n+1\right)^2H_{j}(F^{(n+1)}_{1}).\label{e-6F6n}
\end{gather}
The initial values are
\begin{align*}
 H_{0}(F^{(n+1)}_{1})=&1,\quad H_{1}(F^{(n+1)}_{1})=-\frac32\,\cdot{\frac { \left( n+2 \right)  \left( 2\,n+3 \right) }{n+1}},\quad H_{2}(F^{(n+1)}_{1})=-{\frac { \left( n+2 \right) ^{2}}{ \left( n+1 \right) ^{2}}}.
\end{align*}
By the above initial values, \eqref{e-6F0} and \eqref{e-6F6n}, we obtain the theorem.
\end{proof}

\subsection{The results of $H_n(F(x,8))$}
For $r=8$, we have the functional equation
\begin{align*}
   F(x,8)=\frac{1}{1-8\,x+9\,{x}^{2}-2\,x^3-{x}^{4} F(x,8)}.
\end{align*}
\begin{proof}[Proof of Theorem \ref{example-F8}]
We apply Proposition \ref{xinu(0,i,ii)} to $F_0:=F(x,8)$ by repeatedly using the transformation $\tau$.
This results in a shifted periodic continued fractions of order 4:
\begin{align*}
  F_0(x)& \mathop{\longrightarrow}\limits^\tau F_1^{(1)}(x)\mathop{\longrightarrow}\limits^\tau F^{(1)}_{2}(x)\mathop{\longrightarrow}\limits^\tau F^{(1)}_{3}(x)\mathop{\longrightarrow}\limits^\tau F^{(1)}_{4}(x)\mathop{\longrightarrow}\limits^\tau F^{(1)}_{5}(x)= F_2^{(1)}(x)\cdots.
\end{align*}
We obtain
\begin{align}
  H_k(F_0)=H_{k-1}(F_1^{(1)}).\label{e-8F0}
\end{align}
For $p\geq1$. Computer experiment suggests us to define
\begin{align*}
&F^{(p)}_{1} =-\frac1{45}(2025\,{x}^{6}+10800\,{x}^{5} \left( p-1 \right)
 \left( p+1 \right) +540\,{x}^{4} \left( 64\,{p}^{2}-41 \right)
 \left( p-1 \right)  \left( p+1 \right) \\
 &+120\,{x}^{3} \left( 1024\,{p}
^{5}-1024\,{p}^{4}-944\,{p}^{3}+944\,{p}^{2}+307\,p-172 \right)
 \left( p+1 \right)+{x}^{2}\left( p+1 \right)\\
 & \left( 65536\,{p}^{7}-65536\,{p}^{6}-
251904\,{p}^{5}+251904\,{p}^{4}+88464\,{p}^{3}-174864\,{p}^{2}-52621\,
p+7396 \right)   \\
&-360\,x \left( 128\,{p}^{2}-192\,p-
101 \right)  \left( 2\,p+1 \right)  \left( p+1 \right) p+90\, \left( 2
\,p+1 \right)  \left( p+1 \right) p \left( 64\,{p}^{2}-96\,p-43
 \right) )/\\
 &\left[ \left( 45\,F^{(p)}_{1}{x}^{2}p+90\,{x}^{4}+ \left( -480\,{p}^{2}-240
 \right) {x}^{3}+ \left( -512\,{p}^{4}+1240\,{p}^{2}+172 \right) {x}^{
2}+360\,xp-45\,p \right) p\right].
\end{align*}
Then the results can be summarized as follows:
\begin{align*}
  H_{k-1}(F^{(p)}_{1})=&\left(-\frac{2} {45}\frac{\left({1+2\,p}\right)\left({1+p}\right)\left({43+96\,p-64\,{p}^{2}}\right)} {p}\right)^{k-1} H_{k-2}(F^{(p)}_{2}), \\
H_{k-2}(F^{(p)}_{2})=&\left(-\frac{45} {2}\frac{117+224\,p+64\,{p}^{2}} {\left({1+2\,p}\right)\left({43+96\,p-64\,{p}^{2}}\right)^{2}}\right)^{k-2}H_{k-3}(F^{(p)}_{3}),\\
 H_{k-3}(F^{(p)}_{3})=&\left(-\frac{45} {2}\frac{43+96\,p-64\,{p}^{2}} {\left({1+2\,p}\right)\left({117+224\,p+64\,{p}^{2}}\right)^{2}}\right)^{k-3}H_{k-4}(F^{(p)}_{4}),\\
 H_{k-4}(F_{4}^{(p)})=&\left(-\frac{2} {45}\frac{\left({1+2\,p}\right)p \left({117+224\,p+64\,{p}^{2}}\right)} {1+p}\right)^{k-4} H_{k-5}(F^{(p+1)}_{1}).
\end{align*}
Combining the above formulas gives the recursion
\[H_{k-1}(F_1^{(p)})=\frac{\left({1+p}\right)^{3}} {{p}^{3}}H_{k-5}(F^{(p+1)}_{1}).\]
Let $k-1=4n+j$, where $0\leq j< 4$. We then deduce that
\begin{gather}
  H_{4n+j}(F_1^{(1)})=\left(1+n\right)^{3}H_{j}(F^{(n+1)}_{1}).\label{e-8F8n}
\end{gather}
The initial values are
\begin{align*}
 H_{0}(F^{(n+1)}_{1})=&1,\qquad H_{1}(F^{(n+1)}_{1})=-\frac{2} {45}\frac{\left({2+n}\right)\left({3+2\,n}\right)\left({75-32\,n-64\,{n}^{2}}\right)} {1+n},\\
H_{2}(F^{(n+1)}_{1})=&-\frac{2} {45}\frac{\left({2+n}\right)^{2}\left({3+2\,n}\right)\left({405+352\,n+64\,{n}^{2}}\right)} {\left({1+n}\right)^{2}},\ H_{3}(F^{(n+1)}_{1})=\frac{\left({2+n}\right)^{3}} {\left({1+n}\right)^{3}}.
\end{align*}
Then the theorem follows by the above initial values, \eqref{e-8F0} and \eqref{e-8F8n}.
\end{proof}

\subsection{A summary} Similarly, we have computed $H_n(f(x,r))$ for even $r$ up to $30$, and found a polynomial characterization as in Table \ref{table-even}. This leads to Conjecture \ref{conj-even}.

\begin{table}[h]  
  \centering  
  \label{table5} 
    {\begin{tabular}{|c|c|c|c|c|c|c|c|c|c|c|c|c|}   
    \hline
      {\color{red}{\diagbox[dir=SE]{{\color{black}$r$}}{{\color{black}$i$}}}} &     $0$     & $1$  &  $2$   & $3$   &   $4$  &   $5$   &   $6$   &   $7$   &  $8$   &  $9$  &   $10$ \\
      \hline
     $r=4$ & 1 & 1 &  &  &  &  &  &  &  &  &    \\
  \hline
  $r=6$ & 2 & 2 &3 &  &  &  &  &  & &  &   \\
  \hline
  $r=8$ & 3 & 3 & 6 & 6&  &  &  &  &  &  &   \\
  \hline
  $r=10$ & 4 & 4 & 9 & \textbf{10} & 9 &   &  &  &  &  & \\
  \hline
  $r=12$ &5& 5 & 12 & 15 &15 & 12 &  &  &  & & \\
  \hline
  $r=14$ &6 & 6 & 15 & 20 &\textbf{21}&20 & 15 && & & \\
  \hline
  $r=16$ &7 & 7 & 18& 25 &28&28 &25& 18 &  & &   \\
  \hline
  $r=18$ &8 & 8 & 21& 30 &35&\textbf{36} & 35& 30& 21 &  &  \\
  \hline
  $r=20$ &9 & 9 & 24& 35&42&45 & 45& 42& 35 & 24 &  \\
  \hline
  $r=22$ &10 & 10& 27&40&49&54 & \textbf{55}& 54& 49 & 40 &27 \\
  \hline
    \end{tabular}
  }\vspace{4mm}
  \caption{Up to a sign, $H_{tn+i}(F(x,r))$ is a polynomial in $n$,
    with the corresponding degree given in the table.\label{table-even}}
\end{table}

\section{On Hankel determinants $\det\left(\binom{2i+2j+r}{i+j}\right)_{0\le i,j\le n-1}$}\label{sec-HG}
Our main objective here is to evaluate the Hankel determinants
$\det\left(\binom{2i+2j+r}{i+j}\right)_{0\le i,j\le n-1}$.

\subsection{Some formulas}

The generating function of $\binom{2n+r}{n}$ is
\begin{align}
 G(x,r)= \sum_{n\ge 0}\binom{2n+r}{n} x^n = \frac{C(x)^r}{1-2xC(x)}=\frac{C(x)^r}{\sqrt{1-4x}}.
\end{align}
The cases $r=0,1$ are well-known.
\begin{align}
 H_n(G(x,0))&= 2^{n-1}, \qquad
H_n(G(x,1))=1.
\end{align}
Egecioglu, Redmond, and Ryavec computed $H_n(G(x,2))$ and
$H_n(G(x,3))$ and stated some conjectures for $r > 3$. See \cite{ Omer Egecioglu, Omer Egecioglu 0}.
We list the formulas for $r\le 8$ as follows. The readers will find that the patterns differ for even and odd $r$'s.
\begin{thm}\cite{Omer Egecioglu}\label{example-G3}
  \begin{align*}
    H_{3n}(G(x,3)) &=H_{3n+1}(G(x,3))=2n+1,\quad H_{3n+2}(G(x,3)) =-4(n+1).
  \end{align*}
\end{thm}
\begin{thm} \label{example-G5}
  \begin{align*}
    H_{5n}( G(x,5)) &=H_{5n+1}( G(x,5))=(2n+1)^2,\\
     H_{5n+2}( G(x,5)) &=-\frac13\, \left( 50\,{n}^{2}+89\,n+39 \right)  \left( 2\,n+1 \right),\\
      H_{5n+3}( G(x,5)) &=-16(n+1 )^2,\\
      H_{5n+4}( G(x,5)) &=\frac13\, \left( 100\,{n}^{2}+272\,n+183 \right)  \left( n+1 \right) .
  \end{align*}
\end{thm}

\begin{thm}[Explicitly conjectured in \cite{Omer Egecioglu}]\label{example-G7}
  \begin{align*}
 H_{7n}(G(x,7))=& H_{7n+1}(G(x,7))= \left( 2\,n+1 \right) ^{3},\\
  H_{7n+2}(G(x,7))=&{\frac {1}{90}}\, { \left( n+1 \right)  \left( 9604\,{n}^{3}+
9604\,{n}^{2}-1323\,n-2340 \right) }(2\,n+1)^2,\\
H_{7n+3}(G(x,7))=&-\frac1{45}\, { \left( 19208\,{n}^{3}+67228\,{n}^{2}+70854\,n+23445
 \right)  \left( n+1 \right) ^{2}} \left( 2\,n+1 \right),\\
H_{7n+4}(G(x,7))=&64\, { \left( n+1 \right) ^{3}},\\
H_{7n+5}(G(x,7))=&\frac1{45}\, { \left( n+1 \right) ^{2} \left( 38416\,{n}^{4}+153664\,{
n}^{3}+208936\,{n}^{2}+103344\,n+9045 \right) },\\
 H_{7n+6}(G(x,7))=&-{\frac {1}{90}}\, { \left( 9604\,{n}^{3}+48020\,{n}^{2}+75509\,
n+38110 \right)  \left( 3+2\,n \right) ^{2} \left( n+1 \right) }.
\end{align*}
\end{thm}
\begin{thm}\cite{Omer Egecioglu 0}\label{example-G2}
  \begin{align*}
    H_{2n}( G(x,2))=H_{2n+1}( G(x,2))=(-1)^{n}.
  \end{align*}
\end{thm}
\begin{thm}\cite{J.Cigler-Mike}\label{example-G4}
  \begin{align*}
    H_{4n}( G(x,4)) &=H_{4n+1}( G(x,4))=1,\\
     H_{4n+2}( G(x,4)) &=-H_{4n+3}( G(x,4))=-8(n+1).
  \end{align*}
\end{thm}
\begin{thm}\label{example-G6}
  \begin{align*}
    H_{6n}( G(x,6)) &=H_{6n+1}( G(x,6))=(-1)^n,\\
     H_{6n+2}( G(x,6)) &=(-1)^n\left( n+1 \right)  \left( 144\,{n}^{2}+72\,n-19 \right),\\
      H_{6n+3}( G(x,6)) &=- H_{6n+4}( G(x,6))=(-1)^{n+1}144(n+1)^2,\\
        H_{6n+5}( G(x,6)) &=(-1)^{n}\left( n+1 \right)  \left( 144\,{n}^{2}+504\,n+413 \right).
  \end{align*}
\end{thm}
\begin{thm}\label{example-G8}
  \begin{align*}
    H_{8n}(G(x,8)) &=H_{8n+1}=1,\\
     H_{8n+2}(G(x,8)) &=-\frac2{15}\, \left( n+1 \right)  \left( 256\,{n}^{2}+192\,n+15 \right)
 \left( 256\,{n}^{2}+192\,n+17 \right),\\
      H_{8n+3}(G(x,8)) &={\frac {16}{45}}\, \left( 65536\,{n}^{4}+262144\,{n}^{3}+272896\,{n}^{
2}+79104\,n-3915 \right)  \left( n+1 \right) ^{2},\\
      H_{8n+4}(G(x,8)) &=-H_{8n+5}(G(x,8))=4096\, \left( n+1 \right) ^{3}=(4\times 4)^3\, \left( n+1 \right) ^{3},\\
      H_{8n+6}(G(x,8))&={\frac {16}{45}}\, \left( 65536\,{n}^{4}+262144\,{n}^{3}+272896\,{n}^{
2}-36096\,n-119115 \right)  \left( n+1 \right) ^{2},\\
H_{8n+7}(G(x,8))&=\frac2{15}\, \left( n+1 \right)  \left( 256\,{n}^{2}+832\,n+655 \right)
 \left( 256\,{n}^{2}+832\,n+657 \right).
  \end{align*}
\end{thm}
For general $r$, the following partial results were proved in \cite{J.Cigler-Mike}.
\begin{thm}\cite[Theorems 7.3 and 7.6]{J.Cigler-Mike}\label{Ciglernew-conj-odd}
  For odd positive integer $r=2t+1$, we have
  \begin{align*}
    H_{(2t+1)n}(G(x,r)) & = H_{(2t+1)n+1}(G(x,r))=(2n+1)^t, \\
     H_{(2t+1)n+t+1}(G(x,r))& =(-1)^{\binom{t+1}{2}}4^t(n+1)^t.
  \end{align*}
   For even positive integer $r=2t$, we have
  \begin{align*}
    H_{2tn}(G(x,r)) & = H_{2tn+1}(G(x,r))=(-1)^{tn}, \\
     H_{2tn+t}(G(x,r))&= H_{2tn+t+1}(G(x,r)) =(-1)^{n\binom{2t}{2}}(4t)^{t-1}(n+1)^{t-1}.
  \end{align*}
\end{thm}
Note that this theorem includes $H_{5n+i}(F(x,5))$ for $0\le i\le 3$ as special cases.

Our method works for these Hankel determinants for specific $r$. Firstly, for odd $r$, the first several functional equations are as follows.
\begin{gather*}
 G(x,1) = \frac1{\left( 1-4\,x \right) \left( x G(x,1)+1\right) },\\
G(x,3)= \frac1 {\left(  1-4\,x \right) \left({x}^{3} G(x,3)- \left( x-1 \right) \right) },\\
G(x,5)= \frac1 {\left(  1-4\,x \right) \left({x}^{5} G(x,5)+ \left( {x}^{2}-3\,x+1
 \right) \right) },\\
G(x,7)= \frac1 {\left(  1-4\,x \right) \left({x}^{7} G(x,7)- \left( {x}^{3}-6\,{x}^{
2}+5\,x-1 \right)\right)  },\\
G(x,9)= \frac1 {\left(  1-4\,x \right) \left({x}^{9} G(x,9)+ \left( {x}^{4}-10\,{x}
^{3}+15\,{x}^{2}-7\,x+1 \right)\right) },\\
G(x,11)= \frac1 {\left(  1-4\,x \right) \left( {x}^{11}G(x,11)- \left( {x}^{5}-15\,{x
}^{4}+35\,{x}^{3}-28\,{x}^{2}+9\,x-1 \right) \right)
 },\\
G(x,13)= \frac1 {\left(  1-4\,x \right) \left({x}^{13} G(x,13)+ \left( {x}^{6}-21\,{x
}^{5}+70\,{x}^{4}-84\,{x}^{3}+45\,{x}^{2}-11\,x+1 \right) \right) }.
\end{gather*}
For even $r$,  the first several functional equations are as follows.
\begin{gather*}
  G(x,2)=  \frac1 {\left(  1-4\,x \right) \left( {x}^{2}G(x,2)+1\right) },\\
  G(x,4)=  \frac1 {\left(  1-4\,x \right) \left({x}^{4} G(x,4)-\left( 2\,x-1 \right)
 \right)  },\\
 G(x,6)= \frac1 {\left(  1-4\,x \right) \left({x}^{6}G(x,6)+ \left( 3\,{x}^{2}-4\,x+
1 \right) \right) },\\
G(x,8)= \frac1 {\left(  1-4\,x \right) \left({x}^{8} G(x,8)- \left( 4\,{x}^{3}-10\,{
x}^{2}+6\,x-1 \right)\right)  },\\
G(x,10)= \frac1 {\left(  1-4\,x \right) \left({x}^{10}G(x,10)+ \left( 5\,{x}^{4}-20
\,{x}^{3}+21\,{x}^{2}-8\,x+1 \right)\right) },\\
G(x,12)= \frac1 {\left(  1-4\,x \right) \left({x}^{12}G(x,12)-  \left( 6\,{x}^{5}-35
\,{x}^{4}+56\,{x}^{3}-36\,{x}^{2}+10\,x-1 \right)\right) }.
\end{gather*}
Just like that of $F(x,r)$, we guess and prove explicit functional equations of $G(x,r)$ for general $r$. See Appendix \ref{g-e-proof} for details.

\subsection{Proof for odd $r=2t+1$}
For $r=3$, we have the functional equation
\begin{align*}
  G(x,3)= -\frac1 { \left( 4\,{x}^{4}-{x}^{3} \right)  G(x,3)-4\,{x}^{2}+5\,x-1
 }.
\end{align*}
\begin{proof}[Proof of Theorem \ref{example-G3}]
We apply Proposition \ref{xinu(0,i,ii)} to $G_0:=G(x,3)$ by repeatedly using the transformation $\tau$.
This results in a shifted periodic continued fractions of order 3:
\begin{align*}
  G_0(x)& \mathop{\longrightarrow}\limits^\tau G_1^{(1)}(x)\mathop{\longrightarrow}\limits^\tau G^{(1)}_{2}(x)\mathop{\longrightarrow}\limits^\tau G^{(1)}_{3}(x)\mathop{\longrightarrow}\limits^\tau G^{(1)}_{4}(x)= G_2^{(1)}(x)\cdots.
\end{align*}
We obtain
\begin{align}
  H_k(G_0)=H_{k-1}(G_1^{(1)}).\label{e-3G0}
\end{align}
For $p\geq1$. Computer experiment suggests us to define
\begin{align*}
  G^{(p)}_{1} =&-{\frac {4\,{x}^{2}+ \left( 18\,p+1 \right) x \left( 2\,p-1 \right) -4
\,p \left( 2\,p-1 \right) }{ \left( G^{(p)}_{1}{x}^{2} \left( 2\,p-1 \right) +4
\,{x}^{2}+ \left( 10\,p-5 \right) x-2\,p+1 \right)  \left( 2\,p-1
 \right) }}.
\end{align*}
Then the results can be summarized as follows:
\begin{align*}
  H_{k-1}(G^{(p)}_{1})=&\left(-\frac{4p} {2p-1}\right)^{k-1} H_{k-2}(G^{(p)}_{2}), \\
H_{k-2}(G^{(p)}_{2})=&\left(\frac{1} {16}\frac{4p^2-1} {p^2}\right)^{k-2}H_{k-3}(G^{(p)}_{3}),\\
 H_{k-3}(G_{3}^{(p)})=&\left(-\frac{4p} {2p+1}\right)^{k-3} H_{k-4}(G_{1}^{(p+1)}).
\end{align*}
Combining the above formulas gives the recursion
\[H_{k-1}(G_1^{(p)})=\frac{2p+1}{2p-1}H_{k-4}(G^{(p+1)}_{1}).\]
Let $k-1=3n+j$, where $0\leq j<3$. We then deduce that
\begin{gather}
  H_{3n+j}(G_1^{(1)})=(2n+1)H_{j}(G^{(n+1)}_{1}).\label{e-3G3n}
\end{gather}
The initial values are
\begin{align*}
 H_{0}(G^{(n+1)}_{1})=&1, \qquad H_{1}(G^{(n+1)}_{1})=-\frac{4\left(n+1\right)} {2n+1},\qquad
H_{2}(G^{(n+1)}_{1})=\frac{2n+3}{2n+1}.
\end{align*}
Then the theorem follows by the above initial values, \eqref{e-3G0} and \eqref{e-3G3n}.
\end{proof}

\medskip
For $r=5$, we have the functional equation
\begin{align*}
  G(x,5)= -\frac1 { \left( 4\,{x}^{6}-{x}^{5} \right)  G(x,5)+4\,x^3-13\,{x}^{2}+7\,x-1
 }.
\end{align*}
\begin{proof}[Proof of Theorem \ref{example-G5}]
We apply Proposition \ref{xinu(0,i,ii)} to $G_0:=G(x,5)$ by repeatedly using the transformation $\tau$.
This results in a shifted periodic continued fractions of order 5:
\begin{align*}
  G_0(x)& \mathop{\longrightarrow}\limits^\tau G_1^{(1)}(x)\mathop{\longrightarrow}\limits^\tau G^{(1)}_{2}(x)\mathop{\longrightarrow}\limits^\tau G^{(1)}_{3}(x)\mathop{\longrightarrow}\limits^\tau G^{(1)}_{4}(x)\mathop{\longrightarrow}\limits^\tau G^{(1)}_{5}(x)\mathop{\longrightarrow}\limits^\tau G^{(1)}_{6}(x)=  G_1^{(2)}(x)\cdots.
\end{align*}
We obtain
\begin{align}
  H_k(G_0)=H_{k-1}(G_1^{(1)}).\label{e-5G0}
\end{align}
For $p\geq1$. Computer experiment suggests us to define
\begin{align*}
  G^{(p)}_{1} =&-\frac13\,{\frac {36\,{x}^{4}+ \mathcal{P}_2(p){x}^{3
}+\mathcal{P}_4(p) {x}^{2}+\mathcal{P}_3(p) x -\mathcal{P}_3(p) }{ \left(  \left( 6\,p-3 \right) {x
}^{2} G^{(p)}_{1}-12\,{x}^{3}+ \mathcal{P}_2(p) {x}^{2}+
 \left( 42\,p-21 \right) x-6\,p+3 \right)  \left( 2\,p-1 \right) }}.
\end{align*}
\textbf{Note:} Here and what follows, we use $\mathcal{P}_d(p)$ to denote a polynomial in $p$ of degree $d$. The contents of these
polynomials are not important enough to be printed.
Then the results can be summarized as follows:
\begin{align*}
  H_{k-1}(G^{(p)}_{1})=&\left(-\frac13\,{\frac {p \left( 50\,p-11 \right) }{2\,p-1}}\right)^{k-1} H_{k-2}(G^{(p)}_{2}), \\
H_{k-2}(G^{(p)}_{2})=&\left(-\frac{144}{2500\,{p}^{2}-1100\,p+121 }\right)^{k-2}H_{k-3}(G^{(p)}_{3}),\\
 H_{k-3}(G_{3}^{(p)})=&\left(-{\frac {1}{2304}}\,{\frac { \left( 100\,{p}^{2}+72\,p+11 \right)
 \left( 50\,p-11 \right)  \left( 2\,p-1 \right) }{{p}^{2}}}\right)^{k-3} H_{k-4}(G_{4}^{(p)}),\\
 H_{k-4}(G_{4}^{(p)})=&\left(-\frac{144}{2500\,{p}^{2}+1100\,p+121 }\right)^{k-4} H_{k-5}(G_{5}^{(p)}),\\
 H_{k-5}(G_{5}^{(p)})=&\left(\frac13\,{\frac { \left( 50\,p+11 \right) p}{2\,p+1}} \right)^{k-5} H_{k-6}(G_{1}^{(p+1)}).
\end{align*}
Combining the above formulas gives the recursion
\[H_{k-1}(G_1^{(p)})=\left(\frac{2p+1}{2p-1}\right)^2H_{k-6}(G^{(p+1)}_{1}).\]
Let $k-1=5n+j$, where $0\leq j<5$. We then deduce that
\begin{gather}
  H_{5n+j}(G_1^{(1)})=(2n+1)^2H_{j}(G^{(n+1)}_{1}).\label{e-5G5n}
\end{gather}
The initial values are
\begin{align*}
 H_{0}(G^{(n+1)}_{1})=&1,\quad H_{1}(G^{(n+1)}_{1})=-\frac13\,{\frac {50\,{n}^{2}+89\,n+39}{2\,n+1}},\quad H_{2}(G^{(n+1)}_{1})=-16\,{\frac { \left( n+1 \right) ^{2}}{ \left( 2\,n+1 \right) ^{2}}},\\
H_{3}(G^{(n+1)}_{1})=&\frac13\,{\frac { \left( n+1 \right)  \left( 100\,{n}^{2}+272\,n+183\right) }{ \left( 2\,n+1 \right) ^{2}}},\qquad H_{4}(G^{(n+1)}_{1})=\frac{(2n+3)^2}{(2n+1)^2}.
\end{align*}
Then the theorem follows by the above initial values, \eqref{e-5G0} and \eqref{e-5G5n}.
\end{proof}

\medskip
For $r=7$, we have the functional equation
\begin{align*}
 G(x,7)= \frac1 {\left(  1-4\,x \right) \left({x}^{7} G(x,7)- \left( {x}^{3}-6\,{x}^{
2}+5\,x-1 \right)\right)  }.
\end{align*}
\begin{proof}[Proof of Theorem \ref{example-G7}]
We apply Proposition \ref{xinu(0,i,ii)} to $G_0:=G(x,7)$ by repeatedly using the transformation $\tau$.
This results in a shifted periodic continued fractions of order 7:
\begin{align*}
  G_0(x)& \mathop{\longrightarrow}\limits^\tau G_1^{(1)}(x)\mathop{\longrightarrow}\limits^\tau G^{(1)}_{2}(x)\mathop{\longrightarrow}\limits^\tau G^{(1)}_{3}(x)\mathop{\longrightarrow}\limits^\tau G^{(1)}_{4}(x)\mathop{\longrightarrow}\limits^\tau G^{(1)}_{5}(x)\\
  &\mathop{\longrightarrow}\limits^\tau G^{(1)}_{6}(x)\mathop{\longrightarrow}\limits^\tau G^{(1)}_{7}(x)\mathop{\longrightarrow}\limits^\tau G^{(1)}_{8}(x)= G_1^{(2)}(x)\cdots.
\end{align*}
We obtain
\begin{align}
  H_k(G_0)=H_{k-1}(G_1^{(1)}).\label{e-7G0}
\end{align}
For $p\geq1$. Computer experiment suggests us to define
\begin{align*}
  G^{(p)}_{1} =&-{\frac {1}{180}}\,{\frac {32400\,{x}^{6}+ \mathcal{P}_2(p) {x}^{5}+\mathcal{P}_4(p) {x}^{4}+\mathcal{P}_6(p) {x}^{3}+
 \mathcal{P}_8(p) {x}^{2}-\mathcal{P}_5(p) x+\mathcal{P}_5(p) }{ \left((90p-45)x^2G^{(p)}_{1}+ 180\,{x}^{4}+ \mathcal{P}_2(p) {x}^{3}+ \mathcal{P}_4(p) {x}^{2}+\mathcal{P}_1(p)x-90\,p+45 \right)  \left( 2\,p-1 \right) }}.
\end{align*}
Then the results can be summarized as follows:
\begin{align}
  H_{k-1}(G^{(p)}_{1})=&\left({\frac {1}{90}}\,{\frac {p \left( 9604\,{p}^{3}-19208\,{p}^{2}+8281\,p
-1017 \right) }{2\,p-1}}\right)^{k-1} H_{k-2}(G^{(p)}_{2}), \nonumber\\
H_{k-2}(G^{(p)}_{2})=&\left(-180\,{\frac {19208\,{p}^{3}+9604\,{p}^{2}-5978\,p+611}{ \left( 9604\,
{p}^{3}-19208\,{p}^{2}+8281\,p-1017 \right) ^{2}}}\right)^{k-2}H_{k-3}(G^{(p)}_{3}),\nonumber\\
 H_{k-3}(G_{3}^{(p)})=&\left(1440\,{\frac {9604\,{p}^{3}-19208\,{p}^{2}+8281\,p-1017}{ \left( 19208
\,{p}^{3}+9604\,{p}^{2}-5978\,p+611 \right) ^{2}}}\right)^{k-3} H_{k-4}(G_{4}^{(p)}),\nonumber\\
 H_{k-4}(G_{4}^{(p)})=&\Big(-{\frac {1}{\left(2\right) ^{12} \left(3 \right) ^{4} \left(5\right) ^{2}}}\, \left( 2\,p+1 \right)  \left( 19208\,{
p}^{3}-9604\,{p}^{2}-5978\,p-611 \right)  \left( 2\,p-1 \right)\nonumber\\
& \left( 19208\,{p}^{3}+9604\,{p}^{2}-5978\,p+611 \right)/{{p}^{2}}\Big)^{k-4} H_{k-5}(G_{5}^{(p)}),\label{com-G7}\\
 H_{k-5}(G^{(p)}_{5})=&\left(-1440\,{\frac {9604\,{p}^{3}+19208\,{p}^{2}+8281\,p+1017}{ \left(
19208\,{p}^{3}-9604\,{p}^{2}-5978\,p-611 \right) ^{2}}}\right)^{k-5} H_{k-6}(G^{(p)}_{6}),\nonumber \\
H_{k-6}(G^{(p)}_{6})=&\left(180\,{\frac {19208\,{p}^{3}-9604\,{p}^{2}-5978\,p-611}{ \left( 9604\,{
p}^{3}+19208\,{p}^{2}+8281\,p+1017 \right) ^{2}}}\right)^{k-6}H_{k-7}(G^{(p)}_{7}),\nonumber\\
 H_{k-7}(G_{7}^{(p)})=&\left(-{\frac {1}{90}}\,{\frac {p \left( 9604\,{p}^{3}+19208\,{p}^{2}+8281\,
p+1017 \right) }{2\,p+1}}\right)^{k-7} H_{k-8}(G_{1}^{(p+1)})\nonumber.
\end{align}
Combining the above formulas gives the recursion
\[H_{k-1}(G_1^{(p)})=\left(\frac{2p+1}{2p-1}\right)^3H_{k-8}(G^{(p+1)}_{1}).\]
Let $k-1=7n+j$, where $0\leq j<7$. We then deduce that
\begin{gather}
  H_{7n+j}(G_1^{(1)})=(2n+1)^3H_{j}(G^{(n+1)}_{1}).\label{e-7G7n}
\end{gather}
The initial values are
\begin{align*}
 H_{0}(G^{(n+1)}_{1})=&1,\qquad H_{1}(G^{(n+1)}_{1})={\frac {1}{90}}\,{\frac { \left( n+1 \right)  \left( 9604\,{n}^{3}+
9604\,{n}^{2}-1323\,n-2340 \right) }{2\,n+1}},\\
H_{2}(G^{(n+1)}_{1})=&-\frac1{45}\,{\frac { \left( 19208\,{n}^{3}+67228\,{n}^{2}+70854\,n+23445
 \right)  \left( n+1 \right) ^{2}}{ \left( 2\,n+1 \right) ^{2}}},\\
H_{3}(G^{(n+1)}_{1})=&64\,{\frac { \left( n+1 \right) ^{3}}{ \left( 2\,n+1 \right) ^{3}}},\\
H_{4}(G^{(n+1)}_{1})=&\frac1{45}\,{\frac { \left( n+1 \right) ^{2} \left( 38416\,{n}^{4}+153664\,{
n}^{3}+208936\,{n}^{2}+103344\,n+9045 \right) }{ \left( 2\,n+1
 \right) ^{3}}},\\
 H_{5}(G^{(n+1)}_{1})=&-{\frac {1}{90}}\,{\frac { \left( 9604\,{n}^{3}+48020\,{n}^{2}+75509\,
n+38110 \right)  \left( 3+2\,n \right) ^{2} \left( n+1 \right) }{
 \left( 2\,n+1 \right) ^{3}}},\\
 H_{6}(G^{(n+1)}_{1})=&{\frac { \left( 3+2\,n \right) ^{3}}{ \left( 2\,n+1 \right) ^{3}}} .
\end{align*}
Then the theorem follows by the above initial values, \eqref{e-7G0} and \eqref{e-7G7n}.
\end{proof}

\subsection{The result of $H_n(G(x,r))$ for $r=2t$}
For $r=2$, we have the functional equation
\begin{align*}
  G(x,2)=  \frac1 {\left(  1-4\,x \right) \left( {x}^{2}G(x,2)+1\right) }.
\end{align*}
\begin{proof}[Proof of Theorem \ref{example-G2}]
Let $G_0:=G(x,2)$, we repeatedly apply Proposition \ref{xinu(0)(ii)} to obtain the following periodic continued fractions.
\begin{align*}
  H_k(G_0)&=H_{k-1}(G_1) & G_1=\frac{1-4x}{x^2G_1+4x-1}\qquad\qquad\\
  H_{k-1}(G_1)&=-H_{k-2}(G_2) & G_2= \frac1 {\left(  1-4\,x \right) \left( {x}^{2}G(x,2)+1\right) }
\end{align*}
Now we see that $G_2=G_0$. So the continued fractions is periodic of order $2$: $G_0 \to G_1 \to G_0.$ By summarizing the above results, we obtain the
recursion $H_k(F_0)=-H_{k-2}(F_0)$, which implies that
$H_{2n+i}(G_0)=(-1)^{n}H_i(G_0), \text{ for }i=0,1.$
The initial values are $H_0(G_0)=1,\ H_1(G_0)=1.$ Then we can get $H_n(G(x,2))$.
\end{proof}

\medskip
For $r=4$, we have the functional equation
\begin{align*}
  G(x,4)= -\frac1 { \left( 4\,{x}^{5}-{x}^{4} \right)  G(x,4)-8\,{x}^{2}+6\,x-1
 }.
\end{align*}
\begin{proof}[Proof of Theorem \ref{example-G4}]
We apply Proposition \ref{xinu(0,i,ii)} to $G_0:=G(x,4)$ by repeatedly using the transformation $\tau$.
This results in a shifted periodic continued fractions of order 4:
\begin{align*}
  G_0(x)& \mathop{\longrightarrow}\limits^\tau G_1^{(1)}(x)\mathop{\longrightarrow}\limits^\tau G^{(1)}_{2}(x)\mathop{\longrightarrow}\limits^\tau G^{(1)}_{3}(x)\mathop{\longrightarrow}\limits^\tau G^{(1)}_{4}(x)= G_1^{(2)}(x)\cdots.
\end{align*}
We obtain
\begin{align}
  H_k(G_0)=H_{k-1}(G_1^{(1)}).\label{e-4G0}
\end{align}
For $p\geq1$. Computer experiment suggests us to define
\begin{align*}
  G^{(p)}_{1} =&-{\frac {4\,{x}^{3}+ \left( 64\,{p}^{2}-64\,p-1 \right) {x}^{2}+48\,px
-8\,p}{{x}^{2}G^{(p)}_{1}-1+ \left( 16\,p-8 \right) {x}^{2}+6\,x}}.
\end{align*}
Then the results can be summarized as follows:
\begin{align*}
  H_{k-1}(G^{(p)}_{1})=&\left(-8p\right)^{k-1} H_{k-2}(G^{(p)}_{2}), \\
H_{k-2}(G^{(p)}_{2})=&\left(\frac{1} {8p}\right)^{k-2}H_{k-3}(G^{(p)}_{4}),\\
 H_{k-3}(G_{3}^{(p)})=&\left(-\frac{1} {8p}\right)^{k-3} H_{k-4}(G_{4}^{(p)}),\\
 H_{k-4}(G_{4}^{(p)})=&\left(8p\right)^{k-4} H_{k-5}(G_{1}^{(p+1)}).
\end{align*}
Combining the above formulas gives the recursion
\[H_{k-1}(G_1^{(p)})=H_{k-5}(G^{(p+1)}_{1}).\]
Let $k-1=4n+j$, where $0\leq j<4$. We then deduce that
\begin{gather}
  H_{4n+j}(G_1^{(1)})=H_{j}(G^{(n+1)}_{1}).\label{e-4G4n}
\end{gather}
The initial values are
\begin{align*}
 H_{0}(G^{(n+1)}_{1})=&1,\quad H_{1}(G^{(n+1)}_{1})=-8\left(n+1\right),\quad H_{2}(G^{(n+1)}_{1})=8\left(n+1\right),\quad H_{3}(G^{(n+1)}_{1})=1.
\end{align*}
Then the theorem follows by the above initial values, \eqref{e-4G0} and \eqref{e-4G4n}.
\end{proof}

\medskip
For $r=6$, we have the functional equation
\begin{align*}
  G(x,6)= -\frac1 { \left( 4\,{x}^{7}-{x}^{6} \right)  G(x,6)+4\,x^3-13\,{x}^{2}+7\,x-1
 }.
\end{align*}
\begin{proof}[Proof of Theorem \ref{example-G6}]
We apply Proposition \ref{xinu(0,i,ii)} to $G_0:=G(x,6)$ by repeatedly using the transformation $\tau$.
This results in a shifted periodic continued fractions of order 6:
\begin{align*}
  G_0(x)& \mathop{\longrightarrow}\limits^\tau G_1^{(1)}(x)\mathop{\longrightarrow}\limits^\tau G^{(1)}_{2}(x)\mathop{\longrightarrow}\limits^\tau G^{(1)}_{3}(x)\mathop{\longrightarrow}\limits^\tau G^{(1)}_{4}(x)\mathop{\longrightarrow}\limits^\tau G^{(1)}_{5}(x)\mathop{\longrightarrow}\limits^\tau G^{(1)}_{6}(x) \\
  &\mathop{\longrightarrow}\limits^\tau G^{(1)}_{7}(x) =  G_1^{(2)}(x)\cdots.
\end{align*}
We obtain
\begin{align}
  H_k(G_0)=H_{k-1}(G_1^{(1)}).\label{e-6G0}
\end{align}
For $p\geq1$. Computer experiment suggests us to define
\begin{align*}
  G^{(p)}_{1} =&-{\frac {4\,{x}^{5}+\mathcal{P}_2 (p) {x}^{4}+\mathcal{P}_4(p) {x}^{3}+\mathcal{P}_6(p) {x}^{2}-\mathcal{P}_3(p) x+\mathcal{P}_3(p) }{ G^{(p)}_{1} {x}^{2}-1+
 \mathcal{P}_1(p) {x}^{3}- \mathcal{P}_3(p) {x}^{2}+8\,x}}.
\end{align*}
Then the results can be summarized as follows.
\begin{align*}
  H_{k-1}(G^{(p)}_{1})=&\left(p \left( 144\,{p}^{2}-216\,p+53 \right) \right)^{k-1} H_{k-2}(G^{(p)}_{2}), \\
H_{k-2}(G^{(p)}_{2})=&\left(-\frac{144}{\left( 144\,{p}^{2}-216\,p+53 \right) ^{2}}\right)^{k-2}H_{k-3}(G^{(p)}_{3}),\\
 H_{k-3}(G_{3}^{(p)})=&\left({\frac {1}{144}}\,{\frac {144\,{p}^{2}-216\,p+53}{p}}\right)^{k-3} H_{k-4}(G_{4}^{(p)}),\\
 H_{k-4}(G_{4}^{(p)})=&\left(-{\frac {1}{144}}\,{\frac {144\,{p}^{2}+216\,p+53}{p}}\right)^{k-4} H_{k-5}(G_{5}^{(p)}),\\
 H_{k-5}(G_{5}^{(p)})=&\left(-\frac{144}{\left( 144\,{p}^{2}+216\,p+53 \right) ^{2}}\right)^{k-5} H_{k-6}(G_{1}^{(p+1)}),\\
 H_{k-6}(G_{6}^{(p)})=&\left(-p \left( 144\,{p}^{2}+216\,p+53 \right)\right)^{k-6} H_{k-7}(G_{1}^{(p+1)}).
\end{align*}
Combining the above formulas gives the recursion
\[H_{k-1}(G_1^{(p)})=-H_{k-7}(G^{(p+1)}_{1}).\]
Let $k-1=6n+j$, where $0\leq j<6$. We then deduce that
\begin{gather}
  H_{6n+j}(G_1^{(1)})=(-1)^nH_{j}(G^{(n+1)}_{1}).\label{e-6G6n}
\end{gather}
The initial values are
\begin{align*}
 H_{0}(G^{(n+1)}_{1})=&1,  H_{1}(G^{(n+1)}_{1})=\left( n+1 \right)  \left( 144\,{n}^{2}+72\,n-19 \right), H_{2}(G^{(n+1)}_{1})=-144(n+1)^2,\\
H_{3}(G^{(n+1)}_{1})=&144(n+1)^2, H_{4}(G^{(n+1)}_{1})=\left( n+1 \right)  \left( 144\,{n}^{2}+504\,n+413 \right) ,H_{5}(G^{(n+1)}_{1})=-1.
\end{align*}
Then the theorem follows by the above initial values, \eqref{e-6G0} and \eqref{e-6G6n}.
\end{proof}

\medskip
For $r=8$, we have the functional equation
\begin{align*}
  G(x,8)= -\frac1 { \left( 4\,{x}^{9}-{x}^{8} \right)  G(x,8)+4\,x^3-13\,{x}^{2}+7\,x-1
 }.
\end{align*}
\begin{proof}[Proof of Theorem \ref{example-G8}]
We apply Proposition \ref{xinu(0,i,ii)} to $G_0:=G(x,8)$ by repeatedly using the transformation $\tau$.
This results in a shifted periodic continued fractions of order 8:
\begin{align*}
  G_0(x)& \mathop{\longrightarrow}\limits^\tau G_1^{(1)}(x)\mathop{\longrightarrow}\limits^\tau G^{(1)}_{2}(x)\mathop{\longrightarrow}\limits^\tau G^{(1)}_{3}(x)\mathop{\longrightarrow}\limits^\tau G^{(1)}_{4}(x)\mathop{\longrightarrow}\limits^\tau G^{(1)}_{5}(x)\mathop{\longrightarrow}\limits^\tau G^{(1)}_{6}(x) \mathop{\longrightarrow}\limits^\tau G^{(1)}_{7}(x) \\
  &\mathop{\longrightarrow}\limits^\tau G^{(1)}_{8}(x) \mathop{\longrightarrow}\limits^\tau G^{(1)}_{9}(x) = G_1^{(2)}(x)\cdots.
\end{align*}
We obtain
\begin{align}
  H_k(G_0)=H_{k-1}(G_1^{(1)}).\label{e-8G0}
\end{align}
For $p\geq1$. Computer experiment suggests us to define
\begin{align*}
  G^{(p)}_{1} =&-\frac1{15}\,{\frac {900\,{x}^{7}+ \mathcal{P}_2(p) {x}^{6}+\mathcal{P}_4(p) {x}^{5}+\mathcal{P}_6(p) {x}^{4}+
\mathcal{P}_8(p){x}^{3}+\mathcal{P}_{10}(p) {x}^{2}+\mathcal{P}_5(p) x-\mathcal{P}_5(p) }{15\,G^{(p)}_{1} {x}^{2}-15+
 \mathcal{P}_1(p) {x}^{4}+\mathcal{P}_3(p) {x}^{3}+\mathcal{P}_5(p) {x}^{2}+150\,x}}.
\end{align*}
Then the results can be summarized as follows:
\begin{align*}
  H_{k-1}(G^{(p)}_{1})=&\left(-\frac2{15}\,p \left( 256\,{p}^{2}-320\,p+79 \right)  \left( 256\,{p}^{2}-
320\,p+81 \right) \right)^{k-1} H_{k-2}(G^{(p)}_{2}), \\
H_{k-2}(G^{(p)}_{2})=&\left(20\,{\frac {65536\,{p}^{4}-120320\,{p}^{2}+57600\,p-6731}{ \left( 256
\,{p}^{2}-320\,p+79 \right) ^{2} \left( 256\,{p}^{2}-320\,p+81
 \right) ^{2}}}\right)^{k-2}H_{k-3}(G^{(p)}_{3}),\\
 H_{k-3}(G_{3}^{(p)})=&\left(-4320\,{\frac { \left( 256\,{p}^{2}-320\,p+79 \right)  \left( 256\,{p}
^{2}-320\,p+81 \right) }{ \left( 65536\,{p}^{4}-120320\,{p}^{2}+57600
\,p-6731 \right) ^{2}}}\right)^{k-3} H_{k-4}(G_{4}^{(p)}),\\
 H_{k-4}(G_{4}^{(p)})=&\left(-{\frac {1}{11520}}\,{\frac {65536\,{p}^{4}-120320\,{p}^{2}+57600\,p-
6731}{p}}\right)^{k-4} H_{k-5}(G_{5}^{(p)}),\\
 H_{k-5}(G_{5}^{(p)})=&\left({\frac {1}{11520}}\,{\frac {65536\,{p}^{4}-120320\,{p}^{2}-57600\,p-
6731}{p}}\right)^{k-5} H_{k-6}(G_{6}^{(p)}),\\
 H_{k-6}(G_{6}^{(p)})=&\left(-4320\,{\frac { \left( 256\,{p}^{2}+320\,p+79 \right)  \left( 256\,{p}
^{2}+320\,p+81 \right) }{ \left( 65536\,{p}^{4}-120320\,{p}^{2}-57600
\,p-6731 \right) ^{2}}}\right)^{k-6} H_{k-7}(G_{7}^{(p)}),\\
 H_{k-7}(G_{7}^{(p)})=&\left(20\,{\frac {65536\,{p}^{4}-120320\,{p}^{2}-57600\,p-6731}{ \left( 256
\,{p}^{2}+320\,p+81 \right) ^{2} \left( 256\,{p}^{2}+320\,p+79
 \right) ^{2}}}\right)^{k-7} H_{k-8}(G_{8}^{(p)}),\\
 H_{k-8}(G_{8}^{(p)})=&\left(\frac2{15}\,p \left( 256\,{p}^{2}+320\,p+79 \right)  \left( 256\,{p}^{2}+320
\,p+81 \right)\right)^{k-8} H_{k-9}(G_{1}^{(p+1)}).
\end{align*}
Combining the above formulas gives the recursion
\[H_{k-1}(G_1^{(p)})=H_{k-9}(G^{(p+1)}_{1}).\]
Let $k-1=8n+j$, where $0\leq j<8$. We then deduce that
\begin{gather}
  H_{8n+j}(G_1^{(1)})=H_{j}(G^{(n+1)}_{1}).\label{e-8G8n}
\end{gather}
The initial values are
\begin{align*}
 H_{0}(G^{(n+1)}_{1})=&1,\quad H_{1}(G^{(n+1)}_{1})=-\frac2{15}\, \left( n+1 \right)  \left( 256\,{n}^{2}+192\,n+15 \right)
 \left( 256\,{n}^{2}+192\,n+17 \right) ,\\
H_{2}(G^{(n+1)}_{1})=&{\frac {16}{45}}\, \left( 65536\,{n}^{4}+262144\,{n}^{3}+272896\,{n}^{
2}+79104\,n-3915 \right)  \left( n+1 \right) ^{2},\\
H_{3}(G^{(n+1)}_{1})=&4096\, \left( n+1 \right) ^{3}, \ H_{4}(G^{(n+1)}_{1})=-4096\, \left( n+1 \right) ^{3} ,\\
  H_{5}(G^{(n+1)}_{1})=&{\frac {16}{45}}\, \left( 65536\,{n}^{4}+262144\,{n}^{3}+272896\,{n}^{
2}-36096\,n-119115 \right)  \left( n+1 \right) ^{2},\\
 H_{6}(G^{(n+1)}_{1})=&\frac2{15}\, \left( n+1 \right)  \left( 256\,{n}^{2}+832\,n+655 \right)
 \left( 256\,{n}^{2}+832\,n+657 \right), \quad H_{7}(G^{(n+1)}_{1})=1.
\end{align*}
Then the theorem follows by the above initial values, \eqref{e-8G0} and \eqref{e-8G8n}.
\end{proof}
We arrive at the following conjecture.
\begin{conj}
For odd positive integer $r$ with $r=2t+1$, we have, for $1\leq j\leq \lfloor \frac{r+3}{4}\rfloor$, $H_{rn+j}(G(x,r))$, $H_{rn+t+j}(G(x,r))$ and $H_{rn+t+2-j}(G(x,r))$, $H_{rn+r+1-j}(G(x,r))$
are all polynomials in $n$ of degree $(2j-1)(t+1-j)$.\\
For even positive integer $r$ with $r=2t$, we have, for $1\leq j\leq t$, $(-1)^{tn}H_{rn+j}(G(x,r))$ and $(-1)^{tn}H_{rn+r+1-j}(G(x,r))$
are both polynomials in $n$ of degree $(j-1)(r+1-2j)$.
\end{conj}
Currently, the conjecture is not hard to be verified for $r\le 26$ by our method.

\section{Appendix: Proof of functional equations \label{s-func-eq}}
\subsection{Functional equations of $F(x,r)=C(x)^r$ }\label{f-e-proof}
We have the following result.
\begin{prop}\label{p-func-F(x,r)}
Let $F(x,r)=C(x)^r$. For odd positive integer $r=2t+1$, we have
\begin{align}\label{e-F(x,r)-odd}
   F(x,2t+1)=-\frac{1}{\left(\sum_{i=0}^{t }\frac{(-1)^{i+1}\binom{2t -i}{i}(2t +1)}{2t -2i+1}x^i\right)+{x}^{2t +1} F(x,r)}.
\end{align}
For even positive integer $r=2t$, we have
 \begin{align}\label{e-F(x,r)-even}
   F(x,2t)=-\frac{1}{\left(\sum_{i=0}^{t }(-1)^{i+1}\cdot\frac{(2t )\binom{2t -i}{i}}{2t-i}\cdot x^{i}\right)+{x}^{2t } F(x,r)}.
\end{align}
\end{prop}
\begin{proof}
We find it sufficient to prove the following Lemma \ref{l-func-F(x,r)}. We only prove the odd case formula
\eqref{e-F(x,r)-odd}. The even case formula \eqref{e-F(x,r)-even} is similar.

By setting $y=-C(x)$ in \eqref{e-(y+1)-odd} below, we obtain
\begin{align*}
  -(C(x)-1)^{2t+1}-1&= -\sum_{i=0}^t (-1)^i \frac{2t+1}{2t-2i+1}\binom{2t-i}{i}  (C(x)-1)^i C(x)^{2t-2i+1},\\
  (xC(x)^2)^{2t+1}+1&= \sum_{i=0}^t (-1)^i \frac{2t+1}{2t-2i+1}\binom{2t-i}{i}  (xC(x)^2)^i C(x)^{2t-2i+1},\\
  x^{2t+1} F(x,2t+1)^2 +1&= \sum_{i=0}^t (-1)^i \frac{2t+1}{2t-2i+1}\binom{2t-i}{i}  x^i F(x,2t+1).
\end{align*}
 This is clearly equivalent to \eqref{e-F(x,r)-odd}.
\end{proof}

\begin{lem} \label{l-func-F(x,r)}
  For odd positive integer $r=2t+1$, we have
  \begin{align}\label{e-(y+1)-odd}
    (y+1)^{2t+1}-1= \sum_{i=0}^t  \frac{2t+1}{2t-2i+1}\binom{2t-i}{i}  (y+1)^i y^{2t-2i+1}.
  \end{align}
  For even positive integer $r=2t$, we have
  \begin{align}\label{e-(y+1)-even}
    (y+1)^{2t}+1= \sum_{i=0}^{t}\frac{2t}{2t-i}\binom{2t -i}{i} (y+1)^i y^{2t-2i} .
  \end{align}
\end{lem}
\begin{rema}
  The coefficients $\frac{2t+1}{2t-2i+1}\binom{2t-i}{i}$ are positive integers, since it is easy to verify that
  $$ \frac{2t+1}{2t-2i+1}\binom{2t-i}{i} =\binom{2t-i}{i}+2\binom{2t-i}{i-1}.$$
\end{rema}

\begin{proof}[Proof of Lemma \ref{l-func-F(x,r)}]
By making the substitution $y\to y-1$, \eqref{e-(y+1)-odd} becomes
\begin{align*}
   y^{2t+1}-1&= \sum_{i=0}^t  \frac{2t+1}{2t-2i+1}\binom{2t-i}{i}  y^i (y-1)^{2t-2i+1}\\
   &= \sum_{i=0}^t  \frac{2t+1}{2t-2i+1}\binom{2t-i}{i}  y^i \sum_{j=0}^{2t-2i+1} (-1)^{2t-2i+1-j} \binom{2t-2i+1}{j} y^j \\
   &= \sum_{m=0}^{2t+1} y^m \sum_{i+j=m, 0\le i\le t, 0\le j \le 2t+1-2i} (-1)^{j-1} \frac{2t+1}{2t-2i+1}\binom{2t-i}{i} \binom{2t-2i+1}{j} \\
   &= \sum_{m=0}^{2t+1} y^m \sum_{0\le i\le \min(t,m,2t+1-m)} (-1)^{m-i-1} \frac{2t+1}{2t-2i+1}\binom{2t-i}{i} \binom{2t-2i+1}{j}.
\end{align*}
By equating coefficients, we need to show that
\begin{align}
  \label{e-m-hyper}
Z(t,m)&=0, \quad \text{ for } 1\le m \le 2t\quad  \text{ and } Z(t,0)=-1, \ Z(t,2t+1)=1, \text{ where }\\
Z(t,m)&=\sum_{0\le i\le \min(t,m,2t+1-m)} (-1)^{m-i-1} \frac{2t+1}{2t-2i+1}\binom{2t-i}{i} \binom{2t-2i+1}{m-i}.
\end{align}

Equation \eqref{e-m-hyper} is a hypergeometric sum identity. By using the EKHAD Maple package of Zeilberger's creative telescoping \cite{Zeil}, we find
\begin{gather}
  \left( 2\,t+1 \right) \! \left( m-2\,t-3 \right)\! Z(t+1,m) \!-\!
 \left( {m}^{2}\!-2\,mt\!-3\,m+2\,t+3 \right) \! \left( m-2\,t-1 \right)\! Z(t,m)=0.\label{Zt,m}
\end{gather}
Now we are ready to prove \eqref{e-m-hyper} by induction on $t$.

The boundary case are easy to verify:
\begin{align*}
  Z(t,0)& \xlongequal {i=0, m=0}(-1)^{-1} \frac{2t+1}{2t+1}\binom{2t}{0} \binom{2t+1}{0}=-1, \\
  Z(t,2t+1)&\xlongequal{i=0, m=2t+1}(-1)^{2t} \frac{2t+1}{2t+1}\binom{2t}{2t+1} \binom{2t+1}{0}=1.
  \end{align*}
The base case when $t=1$ is easy, so we omit it. 
Assume \eqref{e-m-hyper} holds for $t$ by the induction hypothesis, i.e., $Z(t,m)=0$ for $1\leq m\leq 2t$. We need to show that
$Z(t+1,m)=0$ for $1\leq m\leq 2t+2$.

The case $1\leq m \leq 2t$ follows from \eqref{Zt,m} and $Z(t,m)=0$. The case $m=2t+1$ directly follows from \eqref{Zt,m}. The case $m=2t+2$ follows by direct computation:
\begin{gather*}
   Z(t+1,2t+2)= (-1)^{2t+1} \frac{2t+3}{2t+3}\binom{2t+2}{0} \binom{2t+3}{2t+2}+ (-1)^{2t+2} \frac{2t+3}{2t+1}\binom{2t+1}{1} \binom{2t+1}{2t+1}=0.
\end{gather*}
This completes the proof.
\end{proof}
\subsection{Functional equations of $G(x,r)=\frac{C(x)^r}{1-2xC(x)}$}\label{g-e-proof}
We have the following result.
\begin{prop}\label{p-func-G(x,r)}
Let $G(x,r)=\frac{C(x)^r}{1-2xC(x)}=\frac{C(x)^r}{\sqrt{1-4x}}$. For odd integer $r=2t+1$, we have
\begin{align}\label{e-G(x,r)-odd}
   G(x,2t+1)=\frac{1}{1-4x}\cdot\frac1{{x}^{2t +1} G(x,r)+\sum_{i=0}^{t }(-1)^{i}\binom{2t -i}{i}x^i}.
\end{align}
For even integer $r=2t$, we have
 \begin{align}\label{e-G(x,r)-even}
    G(x,2t)=\frac{1}{1-4x}\cdot\frac1{{x}^{2t} G(x,r)+\sum_{i=0}^{t-1 }(-1)^{i}\binom{2t -i-1}{i}x^i}.
\end{align}
\end{prop}
\begin{proof}
The proof is similar to that of Proposition \ref{p-func-F(x,r)}. We  also only prove the odd case formula \eqref{e-G(x,r)-odd}.

We find it sufficient to prove the following Lemma \ref{l-func-G(x,r)}.

By setting $y=-C(x)$ in \eqref{yG-odd}, we obtain
\begin{align*}
 (1-C)^{2t+1}+1=&\sum_{i=0}^t\binom{2t-i}{i}(1-C)^iC^{2t+1-2i}+2\sum_{i=0}^t\binom{2t-i}{i}(1-C)^{i+1}C^{2t-2i},\\
  \!\!\!-(xC^2)^{2t+1}+1=&\sum_{i=0}^t\binom{2t-i}{i}(-1)^i(xC^2)^iC^{2t+1-2i}+2\sum_{i=0}^t\binom{2t-i}{i}(-1)^{i+1}(xC^2)^{i+1}C^{2t-2i},\\
  1-x^{2t+1}C^{4t+2}=&\sum_{i=0}^t\binom{2t-i}{i}(-1)^i(xC^2)^iC^{2t+1-2i}+2\sum_{i=0}^t\binom{2t-i}{i}(-1)^{i+1}(xC^2)^{i+1}C^{2t-2i},\\
 =&\sum_{i=0}^t\binom{2t-i}{i}(-1)^ix^iC^{2t+1}-2\sum_{i=0}^t\binom{2t-i}{i}(-1)^{i}x^{i+1}C^{2t+2},\\
  =&\sum_{i=0}^t(-1)^i\binom{2t-i}{i}x^iC^{2t+1}(1-2xC).
\end{align*}
By dividing both sides of the equation by $(1-4x)$, we obtain:
\begin{gather*}
 \frac1{1-4x} -x^{2t+1}\left(\frac{C^{2t+1}}{\sqrt{1-4x}}\right)^2=\sum_{i=0}^t(-1)^i\binom{2t-i}{i}x^i\frac{C^{2t+1}}{\sqrt{1-4x}},\\
\Leftrightarrow\frac1{1-4x}  -x^{2t+1}G(x,2t+1)^2=\sum_{i=0}^t(-1)^i\binom{2t-i}{i}x^iG(x,2t+1).
\end{gather*}
This is clearly equivalent to \eqref{e-G(x,r)-odd}.
\end{proof}
\begin{lem}\label{l-func-G(x,r)}
  For odd positive integer $r=2t+1$, we have
  \begin{gather}
   (y+1)^{2t+1}+1=-\sum_{i=0}^t\binom{2t-i}{i}(1+y)^iy^{2t+1-2i}+2\sum_{i=0}^t\binom{2t-i}{i}(1+y)^{i+1}y^{2t-2i}.\label{yG-odd}
  \end{gather}
   For even positive integer $r=2t$, we have
  \begin{gather}
   (y+1)^{2t}-1=-\sum_{i=0}^{t-1}\binom{2t-i-1}{i}(1+y)^iy^{2t-2i}+2\sum_{i=0}^{t-1}\binom{2t-i-1}{i}(1+y)^{i+1}y^{2t-2i-1}.\label{yG-even}
  \end{gather}
\end{lem}
\begin{proof}By making the substitution $y\rightarrow y-1$, \eqref{yG-odd} becomes
\begin{align*}
  y^{2t+1}+1&=-\sum_{i=0}^t\binom{2t-i}{i}y^i(y-1)^{2t+1-2i}+2\sum_{i=0}^t\binom{2t-i}{i}y^{i+1}(y-1)^{2t-2i}\\
  &=\sum_{i=0}^t\!\!\binom{2t-i}{i}y^i\!\!\sum_{j=0}^{2t+1-2i}\!\!(-1)^j\binom{2t+1-2i}{j} y^j\!+\!2\sum_{i=0}^t\!\!\binom{2t-i}{i}y^{i+1}\!\!\sum_{j=0}^{2t-2i}\!\!(-1)^j\!\binom{2t-2i}{j}y^j\\
  &=\sum_{m=0}^{2t+1}y^m\!\!\sum_{i+j=m,0\leq i\leq t,0\leq j\leq 2t+1-2i}\!\! \left( -1 \right) ^{m-i}{2\,t-i\choose i} \left( {2\,t+1-2\,i\choose
j}-2\,{2\,t-2\,i\choose j} \right)\\
 &=\sum_{m=0}^{2t+1}y^m\!\!\sum_{0\leq i\leq min(t,m, 2t+1-m)}\!\! \left( -1 \right) ^{m-i}{2\,t-i\choose i} \left( {2\,t+1-2\,i\choose
m-i}-2\,{2\,t-2\,i\choose m-i-1} \right).
\end{align*}
By equating coefficients, we need to show that
\begin{align}
 Z(t,m)&=0, \quad \text{for}\quad 1\leq m\leq 2t \quad Z(t,0)=1,\quad Z(t,2t+1)=1,\quad\text{where},\nonumber \\
  Z(t,m)&=\sum_{0\leq i\leq min(t,m, 2t+1-m)} \left( -1 \right) ^{m-i}{2\,t-i\choose i} \left( {2\,t+1-2\,i\choose
m-i}-2\,{2\,t-2\,i\choose m-i-1} \right).\label{Zm,t-G}
\end{align}
Equation \eqref{Zm,t-G} is a hypergeometric sum identity. By using the EKHAD Maple package of Zeilber's creative telescoping, we find
$$\left( 2\,m-2\,t-1 \right)\left( m-2\,t-3 \right) Z(t+1,m)- \left( {m}^{2}-2\,tm-m-2\,t-3 \right)  \left( m-2\,t-1 \right)Z(t,m)=0.$$
The proof is similar to that of Lemma \ref{l-func-F(x,r)}, and we omit it.
\end{proof}

\end{document}